\documentclass[a4paper,11pt]{article}
\usepackage{amsmath,amssymb,amsthm}
\usepackage{cite}
\usepackage{eucal}
\usepackage{microtype}
\usepackage[colorlinks=true,linkcolor=blue,citecolor=blue,urlcolor=blue]{hyperref}
\allowdisplaybreaks

\hypersetup{
  pdftitle={Sharp Poincar\'e Interpolation Along Wasserstein Geodesics},
  pdfauthor={Bang-Xian Han and Zhuo-Nan Zhu},
  pdfsubject={Sharp Poincar\'e inequalities, Gaussian entropy, and optimal transport},
  pdfkeywords={Poincar\'e inequality, coupled Bochner method, strong log-concavity, optimal transport, Wasserstein geodesic, Gaussian entropy, Gaussian Brunn--Minkowski inequality}
}

\numberwithin{equation}{section}
\theoremstyle{plain}
\newtheorem{theorem}{Theorem}[section]
\newtheorem{corollary}[theorem]{Corollary}
\newtheorem{lemma}[theorem]{Lemma}
\newtheorem{proposition}[theorem]{Proposition}
\theoremstyle{definition}
\newtheorem{definition}[theorem]{Definition}
\newtheorem{question}[theorem]{Question}
\theoremstyle{remark}
\newtheorem{remark}[theorem]{Remark}

\newcommand{\R}{\mathbb{R}}
\renewcommand{\d}{\mathrm d}
\newcommand{\Var}{\operatorname{Var}}
\newcommand{\Cov}{\operatorname{Cov}}
\newcommand{\Ent}{\operatorname{Ent}}
\newcommand{\FI}{\mathrm I}
\newcommand{\diver}{\operatorname{div}}
\newcommand{\tr}{\operatorname{tr}}
\newcommand{\HS}{\mathrm{HS}}
\newcommand{\Hess}{\operatorname{Hess}}
\newcommand{\norm}[1]{\left\lVert #1\right\rVert}
\newcommand{\ip}[2]{\left\langle #1,#2\right\rangle}

\title{\fontsize{14}{17}\selectfont\bfseries
Sharp Poincar\'e interpolation along Wasserstein geodesics}
\author{Bang-Xian Han\thanks{School of Mathematics, Shandong University, Jinan 250100, China. Email: hanbx@sdu.edu.cn}
\and Zhuo-Nan Zhu\thanks{School of Mathematical Sciences, University of Science and Technology of China, Hefei 230026, China. Email: zhuonanzhu@mail.ustc.edu.cn}}
\date{}

\begin{document}

\maketitle

\begin{abstract}
Let $\mu_0$ and $\mu_1$ be $\kappa_0$- and $\kappa_1$-strongly
log-concave probability measures on $\R^n$, and let $(\mu_t)_{t\in[0,1]}$
be their quadratic Wasserstein geodesic.  We prove the sharp Poincar\'e constant estimate
\[
\sqrt{C_P(\mu_t)}
\leq
\frac{1-t}{\sqrt{\kappa_0}}
+
\frac{t}{\sqrt{\kappa_1}}.
\]
The coefficient is optimal for every $t,\kappa_0,\kappa_1$, and the
result remains valid for extended-valued potentials without symmetry
assumptions.  Equality at an interior time holds exactly when the two
endpoints have Gaussian factors in the same direction, with variances
$\kappa_0^{-1}$ and $\kappa_1^{-1}$.  All such directions form a maximal
linear subspace.
This gives an affirmative answer, without symmetry or parity
assumptions, to a question of Aishwarya--Rotem.

The proof develops a coupled Bochner method for the two endpoints.  We
solve a weighted Poisson equation involving the Hessian of the Brenier
potential.  From its solution we construct a Bochner couple, and separate estimates for its two
fields are combined along the interpolation.  No curvature bound is
needed for the intermediate measures.  Applications include
centered Gaussian relative entropy, Gaussian Brunn--Minkowski
inequalities with barycenter terms, and centered HWI, logarithmic Sobolev,
and Talagrand estimates.
\end{abstract}



\setcounter{tocdepth}{1}
\begingroup
\setlength{\parskip}{0pt}
\tableofcontents
\endgroup

\section{Introduction}
\subsection{Background and main results}
We study whether analytic estimates implied by curvature bounds at two
endpoints remain valid along their quadratic Wasserstein geodesic. Convexity
of the intermediate measures cannot be assumed: even for uniform measures on
convex bodies, the support of a Wasserstein midpoint can be nonconvex
\cite{SantambrogioWang}, and the midpoint is then not log-concave.

Optimal transport has long linked interpolation inequalities, entropy,
and curvature. Otto and Villani proved, under curvature hypotheses, the
implication from logarithmic Sobolev to Talagrand inequalities and the HWI
estimate \cite{OttoVillani}.
Cordero-Erausquin, McCann, and Schmuckenschl\"ager established a
Riemannian interpolation inequality through optimal transport
\cite{CorderoMcCannSchmuckenschlager}. Von Renesse and Sturm
characterized lower Ricci curvature by displacement convexity of entropy
and equivalent transport and gradient estimates \cite{vonRenesseSturm}.
On Riemannian manifolds with boundary, the first author subsequently proved
measure-rigidity results for the corresponding curvature-dimension
conditions \cite{HanMeasureRigidity}. These results motivate the study of
spectral inequalities along a prescribed Wasserstein geodesic when
curvature assumptions are imposed only at its endpoints.

We prove that positive curvature bounds at the endpoints yield a
Poincar\'e estimate along their quadratic Wasserstein geodesic. The sharp
estimate is linear in the scale $\sqrt{C_P(\mu)}$, which is
the standard deviation for an isotropic Gaussian. If the endpoint
(Bakry--\'Emery) curvature bounds are $\kappa_0$ and $\kappa_1$, then
\[
\sqrt{C_P(\mu_t)}
\leq
(1-t)\kappa_0^{-1/2}+t\kappa_1^{-1/2}.
\]
Gaussian endpoints attain equality, so the coefficient is optimal. The
argument does not require a lower Hessian bound for the intermediate
potentials.

The proof links the two endpoint Bochner inequalities through one
weighted Poisson equation.  It does not propagate convexity of the
potentials along the geodesic.  The source and target estimates are
proved separately.  This keeps their sharp constants and is also
important for the equality case.  We call this argument the
\emph{coupled Bochner method}.

The basic object is a \emph{Bochner couple}, a pair of vector fields at
the two endpoints.  The fields have the same divergence datum and are
related by $DT=\Hess{\Phi}$, the differential of the Brenier map.  Each field satisfies its own Bochner
estimate.  We first construct this pair from a single scalar Poisson
solution.  We then interpolate the two fields to obtain the estimate at
time $t$.  We refer to this construction and interpolation as
\emph{Bochner coupling}.

For a probability measure $\mu\in\mathcal P(\R^n)$, its Poincar\'e constant
$C_P(\mu)$ is defined as the least $C\in[0,\infty]$ such that
\begin{align*}
 \Var_\mu(f)
 :=\int\left|f-\int f\,\d\mu\right|^2\,\d\mu
 \leq C\int|\nabla f|^2\,\d\mu,\qquad\text{for every }f\in C_c^\infty(\R^n).
\end{align*}
For integrable functions $f$ and $g$, we use
\[
 \Cov_\mu(f,g)
 :=\int\left(f-\int f\,\d\mu\right)
          \left(g-\int g\,\d\mu\right)\,\d\mu.
\]
If $\mu$ has a finite first moment, write its barycenter $b_\mu:=\int x\,\d\mu$. 
If $\mu$ has a finite second moment, its covariance matrix is
\[
 \Cov(\mu):=\int (x-b_\mu)\otimes(x-b_\mu)\,\d\mu.
\]
For a random vector $X$ with law $\mu$, we write
$\Cov(X):=\Cov(\mu)$.

We write $\d\mu=e^{-V}\,\d x$ for a probability measure, absorbing the
normalizing constant into the potential $V$. We assume throughout that
$V:\R^n\to(-\infty,+\infty]$ is proper and lower semicontinuous.
We call $\mu$ log-concave if $V$ is convex.  For $\kappa>0$, we call
$\mu$ $\kappa$-strongly log-concave if $V-\kappa\frac{|x|^2}{2}$ is convex.
In the latter case, the classical Brascamp--Lieb inequality
\cite{BrascampLieb} yields $C_P(\mu)\leq\kappa^{-1}$.

Now let $\mu_0$ and $\mu_1$ be strongly log-concave.  Since $\mu_0$ is
absolutely continuous, Brenier's theorem \cite{Brenier} provides an
optimal transport map $T$ from $\mu_0$ to $\mu_1$. 
The constant-speed $W_2$-geodesic joining $\mu_0$ and $\mu_1$ is given by
\[
F_t=(1-t)I+tT,\qquad
\mu_t=(F_t)_\#\mu_0,\qquad t\in[0,1].
\]
Here $W_2$ denotes the quadratic Wasserstein distance and $I$ is the
identity map. The strong log-concavity assumptions give Brascamp--Lieb bounds at the
two endpoints, but no corresponding lower Hessian bound for the potential
of $\mu_t$ is available a priori.  Thus the Brascamp--Lieb
inequality cannot be applied directly at intermediate times.

The following theorem gives the sharp bound for possibly different
endpoint curvatures, without any symmetry assumption. For $i=0,1$, let $\mu_i$ be
$\kappa_i$-strongly log-concave for some
$\kappa_i>0$.  We allow the corresponding potentials to take the value
$+\infty$; in particular, the result includes Gaussian measures
conditioned on full-dimensional convex sets.

\begin{theorem}[Sharp Poincar\'e interpolation]
\label{thm:main}
For every $t\in[0,1]$, we have
\begin{equation}\label{eq:main}
 \Var_{\mu_t}(f)
 \leq\left(\frac{1-t}{\sqrt{\kappa_0}}
           +\frac{t}{\sqrt{\kappa_1}}\right)^2
       \int|\nabla f|^2\,\d\mu_t,
       \qquad\text{for every }f\in C_c^\infty(\R^n).
\end{equation}
Equivalently,
\[
 \sqrt{C_P(\mu_t)}\leq \frac{1-t}{\sqrt{\kappa_0}}
+\frac{t}{\sqrt{\kappa_1}}.
\]
\end{theorem}

\begin{remark}
The coefficient in \eqref{eq:main} is optimal for every
$t\in[0,1]$ and every $\kappa_0,\kappa_1>0$. Indeed, let
$\mu_i=N(0,\kappa_i^{-1}I)$ for $i=0,1$, where $N(m,\Sigma)$ denotes
the Gaussian law with mean $m$ and covariance matrix $\Sigma$. Then
$T(x)=\sqrt{\kappa_0/\kappa_1}\,x$ is the unique optimal transport map
from $\mu_0$ to $\mu_1$. It follows that
\[
\mu_t
=N\!\left(
0,
\left(\frac{1-t}{\sqrt{\kappa_0}}
+\frac{t}{\sqrt{\kappa_1}}\right)^2I
\right).
\]
Thus the coefficient is optimal because
$C_P(N(0,a^2I))=a^2$.
\end{remark}

\begin{remark}
	A stronger statement would replace $\kappa_i^{-1/2}$ in
	Theorem~\ref{thm:main} by $\sqrt{C_P(\mu_i)}$; that is, it would assert
	\[
	\sqrt{C_P(\mu_t)}
	\leq
	(1-t)\sqrt{C_P(\mu_0)}
	+t\sqrt{C_P(\mu_1)}.
	\]
	This interpolation bound in terms of the endpoint Poincar\'e constants
	fails in general, even for strongly
	log-concave endpoints; see
	Proposition~\ref{prop:actual-CP-counterexample}.
\end{remark} 

We next characterize equality at an interior time.

\begin{theorem}[Equality characterization]
\label{thm:equality}
Fix $t\in(0,1)$ and denote $m_i=\int x\,\d\mu_i$. Then
\[
 C_P(\mu_t)
 =\left(\frac{1-t}{\sqrt{\kappa_0}}
        +\frac{t}{\sqrt{\kappa_1}}\right)^2
\]
if and only if there exists a unit
vector $e\in\R^n$ such that, under the orthogonal decomposition
$\R^n=\R e\oplus e^\perp$, the coordinate map
$\mathcal C_e(x)=(e\cdot x,P_{e^\perp}x)$ satisfies
\begin{equation}\label{eq:factor}
 (\mathcal C_e)_\#\mu_i
 =N\!\left(\ip{e}{m_i},\kappa_i^{-1}\right)\otimes\nu_i, \qquad \nu_i\in\mathcal P(e^\perp)
 \qquad i=0,1.
\end{equation}
In this case the optimal transport splits along the same decomposition and
\[
 C_P(\mu_s)
 =\left(\frac{1-s}{\sqrt{\kappa_0}}
        +\frac{s}{\sqrt{\kappa_1}}\right)^2,
 \qquad\text{for every }s\in[0,1].
\]
\end{theorem}

The restriction $t\in(0,1)$ is essential in the rigidity statement.
At an endpoint, equality constrains only that endpoint and cannot force
a Gaussian factor in the other measure.
The proof also identifies the maximal subspace of directions along
which both endpoints split off the corresponding Gaussian factors;
see Proposition~\ref{prop:maximal-Gaussian-factor}.
In \cite[Question 1.6]{AishwaryaRotem}, Aishwarya and Rotem asked about
the unit Poincar\'e bound along the optimal Wasserstein interpolation
between even $1$-strongly log-concave endpoints.  Their question was
restricted to odd test functions; see also
\cite[Question 1.4]{AishwaryaLi}.  Theorem~\ref{thm:main} proves a
stronger result.  It requires neither symmetry of the endpoints nor a
parity assumption on the test function.  Aishwarya and Rotem proved
their formulation in dimension one \cite[Theorem~4.11]{AishwaryaRotem},
for two Gaussian endpoints, and when one endpoint is the standard Gaussian
\cite[Proposition~4.12]{AishwaryaRotem}.

Aishwarya and Li used the common-source coupling
$(T_0(Z),T_1(Z))$, where $Z$ is standard Gaussian and $T_i$ is the Brenier
map from the Gaussian measure to the $i$-th endpoint. Every interpolant
$((1-t)T_0+tT_1)(Z)$ is a $1$-Lipschitz image of $Z$ and therefore satisfies
the full Poincar\'e inequality with constant $1$
\cite[Remark~1]{AishwaryaLi}. This coupling is generally not optimal
between the two prescribed endpoints. To the best of our knowledge, the
full Poincar\'e bound for the quadratic optimal coupling of the prescribed
endpoints had remained open. Theorem~\ref{thm:main} proves it along the quadratic Wasserstein
geodesic between the prescribed endpoints.

The full Poincar\'e estimate also controls the fluctuation of the
velocity gradient in the second variation of Gaussian relative entropy.
We first remove the translation component determined by the means and
then apply the estimate componentwise. Let $\gamma$ be the standard Gaussian measure
on $\R^n$. For a probability measure $\mu$, define its relative entropy
with respect to $\gamma$ by
\[
 \Ent_\gamma(\mu)
 =\int\rho\log\rho\,\d\gamma
 \quad\text{if }\,\d\mu=\rho\,\d\gamma,
\]
with the convention $0\log0=0$, and set
$\Ent_\gamma(\mu)=+\infty$ otherwise. Define
\[
 \mathcal D(\mu)
 =\Ent_\gamma(\mu)-\frac12|b_\mu|^2.
\]
Thus $\mathcal D(\mu)$ is the \emph{centered Gaussian relative entropy}
of $\mu$; equivalently, it is the Gaussian relative entropy of the
centered translate of $\mu$. For two measures, define
\[
W_{\rm c}^2
 =W_2^2(\mu_0,\mu_1)-|b_{\mu_1}-b_{\mu_0}|^2.
\]

\begin{theorem}[Centered Gaussian relative entropy]
\label{thm:centered-entropy}
Let $\mu_0$ and $\mu_1$ be $1$-strongly log-concave, and let
$(\mu_t)_{0\leq t\leq1}$ be their quadratic Wasserstein geodesic.  Then
$0\leq W_{\rm c}^2\leq2n$, and
$\mathcal D_t=\mathcal D(\mu_t)$ satisfies the following distortion
inequality along the whole geodesic:
\begin{equation}\label{eq:intro-entropy-distortion}
 e^{-\mathcal D_t/n}
 \geq
 \frac{\sin((1-t)\omega)}{\sin\omega}
 e^{-\mathcal D_0/n}
 +\frac{\sin(t\omega)}{\sin\omega}
 e^{-\mathcal D_1/n},\qquad \omega=\sqrt{\frac{2}{n}}\,W_{\rm c}.
\end{equation}
The sine quotients are understood by continuity when $\omega=0$, in
which case they equal $1-t$ and $t$, respectively.
Moreover, equality in \eqref{eq:intro-entropy-distortion} at some
$t\in(0,1)$ holds if and only if $\mu_1$ is a translate of $\mu_0$.
In this case, equality holds for every $t\in[0,1]$. 

In particular, if the endpoints have the same barycenter, then the same statements
hold with $\Ent_\gamma(\mu_t)$ in place of $\mathcal D_t$ and with
$W_2(\mu_0,\mu_1)$ in place of $W_{\rm c}$. In this case, equality at
an interior time holds if and only if $\mu_0=\mu_1$. 
\end{theorem}

\paragraph{Infinitesimal form.}
Under the smooth two-sided Hessian hypotheses of
Section~\ref{sec:smooth}, the proof first establishes
\begin{equation}\label{eq:intro-entropy-differential}
 \mathcal D_t''
 \geq2W_{\rm c}^2+\frac1n(\mathcal D_t')^2.
\end{equation}
This is the differential form of the $(2,n)$-convexity estimate; the
finite-time inequality \eqref{eq:intro-entropy-distortion} follows by a
one-dimensional comparison argument and then by regularization.

When the endpoints have the same barycenter,
Theorem~\ref{thm:centered-entropy} establishes $(2,n)$-convexity
along the Wasserstein geodesic for arbitrary
$1$-strongly log-concave measures. Aishwarya and Rotem obtained the
same differential estimate under a pathwise odd Poincar\'e assumption
\cite[Theorem~4.6]{AishwaryaRotem} and verified this assumption in
dimension one and along two further classes of geodesics
\cite[Theorem~4.11 and Proposition~4.12]{AishwaryaRotem}. For even
	endpoints, Aishwarya and Li obtained an entropy inequality with sine
coefficients along the interpolation of the two Brenier maps transporting
the standard Gaussian measure to the respective endpoints. They also
characterized equality in the resulting affine inequality
\cite[Theorem~1.5 and Remark~1]{AishwaryaLi}. The present theorem
requires neither symmetry nor a pathwise odd Poincar\'e assumption.  It
proves the inequality with sine coefficients along the quadratic
Wasserstein geodesic and characterizes equality.

Theorem~\ref{thm:centered-entropy} also yields a Gaussian
Brunn--Minkowski inequality.  Gardner and Zvavitch conjectured the
dimensional Gaussian Brunn--Minkowski inequality for convex bodies
containing the origin \cite{GardnerZvavitch}.  Nayar and Tkocz disproved
this conjecture in the nonsymmetric case, even under the assumption that
both bodies contain the origin, and proposed its origin-symmetric form
\cite{NayarTkocz}.  Following important local work of Kolesnikov and
Livshyts \cite{KolesnikovLivshyts}, the origin-symmetric conjecture was
proved by Eskenazis and Moschidis \cite{EskenazisMoschidis}.
Corollary~\ref{cor:barycenter-BM} replaces origin symmetry by the
assumption that the conditioned Gaussian measures have the same
barycenter.  Without this assumption, it gives an inequality with
explicit barycenter terms.
For a convex Borel set $K$ with $\gamma(K)>0$, denote
\[
 \gamma_K(A)=\frac{\gamma(A\cap K)}{\gamma(K)}
 \qquad\text{and}\qquad
 b_K=\int x\,\d\gamma_K(x),
\]
the Gaussian measure conditioned on $K$ and its barycenter, respectively.

\begin{corollary}[Gaussian Brunn--Minkowski with barycenter terms]
\label{cor:barycenter-BM}
Let $K_0,K_1\subset\R^n$ be convex Borel sets with
$\gamma(K_i)>0$. For every $t\in[0,1]$, set
$K_t=(1-t)K_0+tK_1$ and
$b_t=(1-t)b_{K_0}+tb_{K_1}$.  Then
\begin{equation}\label{eq:intro-barycenter-BM}
 e^{|b_t|^2/(2n)}\gamma(K_t)^{1/n}
 \geq
 (1-t)e^{|b_{K_0}|^2/(2n)}\gamma(K_0)^{1/n}
 +t e^{|b_{K_1}|^2/(2n)}\gamma(K_1)^{1/n}.
\end{equation}
Moreover, equality in \eqref{eq:intro-barycenter-BM} at some
$t\in(0,1)$ holds if and only if
$\bar K_0=\bar K_1$. In this case, equality holds for every
$t\in[0,1]$. 

In particular, if $b_{K_0}=b_{K_1}$, then
\[
 \gamma(K_t)^{1/n}
 \geq(1-t)\gamma(K_0)^{1/n}+t\gamma(K_1)^{1/n}.
\]
\end{corollary}

The condition $b_{K_0}=b_{K_1}$ does not require symmetry.  For example,
a regular simplex centered at the origin is invariant under its full
symmetry group.  It follows that the Gaussian measure conditioned on
the simplex has barycenter zero.  The simplex is not centrally symmetric when
$n\geq2$.  Section~\ref{sec:Gaussian-applications} gives a functional
version.  It extends the Gaussian Borell--Brascamp--Lieb inequality of
Aishwarya and Li beyond even log-concave functions.  The new assumption
is that the probability measures proportional to $f\,\d\gamma$ and
$g\,\d\gamma$ have the same barycenter.  We also provide HWI,
logarithmic Sobolev, Talagrand, and
Ornstein--Uhlenbeck estimates after centering.

\subsection{Proof strategy: coupled endpoint Bochner estimates}

The Bochner identity is used only at $\mu_0$ and $\mu_1$.  We construct
one vector field at each endpoint and interpolate these fields.  Thus no
curvature estimate is required for the potential of the intermediate
measure.  Treating the endpoints separately gives the optimal
coefficient and makes the equality case accessible.

\paragraph{Step 1: the weighted Poisson equation.}
We first work in the smooth setting. Let $T=\nabla\Phi$ be the Brenier
map and set $H=\Hess{\Phi}$. For a centered function $g\in L^2(\mu_0)$, we
solve
\[
-\diver_{\mu_0}(H^{-1}\nabla u)=g
\]
and define
\[
X=H^{-1}\nabla u,
\qquad
Y\circ T=\nabla u.
\]
The same scalar function $u$ defines both fields, while $H$ relates
them through $Y\circ T=HX$.  Smooth approximations of $(X,Y)$ are
Bochner couples in the sense of Definition~\ref{def:bochner-couple}.
A change-of-variables identity for weighted divergence shows that the
two fields have the same divergence datum.

\paragraph{Step 2: Bochner estimates at the two endpoints.}
At the source, differentiating $HX=\nabla u$ and using the symmetry of
the third derivatives of $\Phi$ shows that
\[
H\,DX \quad\text{is symmetric}.
\]
At the target, the chain rule gives
\[
((DY)\circ T)H=\Hess{u},
\]
so this matrix is also symmetric.  Because $H$ is positive definite,
each identity makes the corresponding trace term in the vector Bochner
formula nonnegative.  Applying that formula at the two endpoints gives
\[
\norm{X}_{L^2(\mu_0)}
\leq \kappa_0^{-1/2}\norm{g}_{L^2(\mu_0)},
\qquad
\norm{Y}_{L^2(\mu_1)}
\leq \kappa_1^{-1/2}\norm{g}_{L^2(\mu_0)}.
\]
Both estimates are sharp.  We keep them separate until the final
interpolation step.

\paragraph{Step 3: interpolation of the two fields.}
Let $F_t=(1-t)I+tT$ and define $q_t$ on $\mu_t$ by
\[
q_t\circ F_t=(1-t)X+t(Y\circ T).
\]
The same change-of-variables formula for weighted divergence yields
\[
-(\diver_{\mu_t}q_t)\circ F_t=g.
\]
The separate endpoint estimates imply
\[
\norm{q_t}_{L^2(\mu_t)}
\leq
\left(
\frac{1-t}{\sqrt{\kappa_0}}
+\frac{t}{\sqrt{\kappa_1}}
\right)
\norm{g}_{L^2(\mu_0)}.
\]
Taking $g=(f-\int f\,\d\mu_t)\circ F_t$ and integrating by parts proves
\eqref{eq:main} in the smooth setting. The general case follows by
regularization and stability of optimal transport.

This interpolation of the Bochner couple is the coupling step.  The
endpoint norms combine linearly before the final Cauchy--Schwarz
inequality.  This gives the sharp coefficient.

For the equality case, take a sequence that approaches equality in the
Poincar\'e estimate.  The two endpoint estimates and the triangle
inequality must then become asymptotically sharp.  The associated scalar
functions converge to affine functions in the same direction at both
endpoints.  Gaussian integration by parts gives the common Gaussian
factorization.  When $\kappa_0=\kappa_1=1$, the full Poincar\'e inequality
also gives a $\Gamma_2$ estimate after the translation component is
removed.  The entropy inequality follows from the standard variation
formulas and a one-dimensional ODE comparison.

\subsection{Relations to other methods and further directions}

\paragraph{Poisson equations, Stein's method, and isoperimetry.}
The construction also has a Stein interpretation.  For a probability
measure $\d\mu=e^{-V}\,\d x$, the weighted divergence $-\diver_\mu$ is a
vector-valued Stein operator.  Given a centered test function
$h$, one seeks a field $q$ satisfying
\[
 -\diver_\mu q=h
\]
and estimates its size or regularity.  Such estimates are called
Stein-factor bounds. Courtade, Fathi, and Pananjady constructed Stein
kernels by a variational argument under a Poincar\'e inequality
\cite{CourtadeFathiPananjady}. Mijoule, Rai\v{c}, Reinert, and Swan
introduced weak Stein equations and obtained weak Stein-factor bounds
from a Poincar\'e inequality \cite{MijouleRaicReinertSwan}. Fathi gave a
different construction of Stein kernels through moment maps
\cite{FathiSteinMoment}. In this terminology,
$\sqrt{C_P(\mu)}$ is the optimal $L^2$ norm of a right inverse of
$-\diver_\mu$ on centered functions.  Our field $q_t$ is an $L^2$ Stein
solution for $\mu_t$.  In the usual application, Stein's method compares
a measure with a reference law.  Here one Poisson solution gives
compatible fields at two prescribed endpoints.  Their Bochner estimates
control the $L^2$ norm of the interpolated field.

Jiang and Koskela used Poisson equations to derive local isoperimetric and
Sobolev inequalities on metric measure spaces from estimates for the
Cheeger--Poisson equation \cite{JiangKoskela}. In particular, an estimate of the form
\[
 \norm{\nabla u}_{L^\infty_{\mathrm{loc}}}
 \lesssim \norm{g}_{L^{Q,1}_{\mathrm{loc}}},
 \qquad \Delta u=g,
\]
combined with integration by parts and the dual characterization of perimeter yields
the isoperimetric inequality with the optimal exponent.  Jiang, Koskela,
and Yang later obtained local isoperimetric inequalities from
Lipschitz regularity of Cheeger-harmonic functions
\cite{JiangKoskelaYang}.  Their argument uses an $L^\infty$ gradient
estimate to control perimeter.  Our argument instead uses sharp $L^2$
estimates for two vector fields, one at each endpoint, to control
variance at the intermediate measure.

\paragraph{Hessian metrics induced by optimal transport and related flows.}

Kolesnikov associates with the Brenier map $T=\nabla\Phi$ the Hessian
metric $g=\Hess{\Phi}$. Its weighted Laplacian is
\[
L_\Phi=\diver_{\mu_0}\bigl((\Hess{\Phi})^{-1}\nabla\bigr),
\]
and the source--target integration-by-parts identity used below is stated
in \cite[Lemma~2.1 and (2.5)]{KolesnikovHessian}. Klartag used a dual
Bochner formula in his moment-map proof of Poincar\'e inequalities
\cite{KlartagMoment}. Kolesnikov and Milman derived Poincar\'e and
logarithmic Sobolev inequalities by choosing Riemannian metrics adapted to
the measure, including Hessian metrics arising from optimal transport
\cite{KolesnikovMilman}.

Klartag and Putterman proved monotonicity of the Poincar\'e constant under
Gaussian convolution, using a time-dependent diffusion in one proof and a
contracting transport in another \cite{KlartagPutterman}. Serres studied
the Poincar\'e constant along the Polchinski flow by dynamic
$\Gamma$-calculus \cite{SerresPolchinski}. Those papers concern paths
generated by the respective flows. Here the path is a prescribed
displacement interpolation, while the Bochner estimates are taken at the
two endpoints and linked by the differential of the Brenier map.

Our argument does not use a curvature bound for the Hessian metric. For
$X=H^{-1}\nabla u$ and $Y\circ T=\nabla u$, the two matrix identities in
Step~2 allow the Euclidean vector Bochner identity to be applied
separately at the source and the target. They give
\[
\kappa_0\norm{X}_{L^2(\mu_0)}^2
\leq\norm{g}_{L^2(\mu_0)}^2,
\qquad
\kappa_1\norm{Y}_{L^2(\mu_1)}^2
\leq\norm{g}_{L^2(\mu_0)}^2.
\]
Combining the two separate estimates only at the final interpolation
step gives the optimal coefficient in Theorem~\ref{thm:main}.  It also
allows us to analyze equality at both endpoints.

The matrix fact used to control $\tr((DX)^2)$ is the following: if $H$ is
positive definite and $H DX$ is symmetric, then $DX$ is similar to a
symmetric matrix and $\tr((DX)^2)\geq0$. The same conjugation appears in
\cite[proof of Lemma~3.1]{KlartagKolesnikov}; a related trace inequality
for differentials of common-source Brenier interpolants appears in
\cite[proof of Theorem~1.5]{AishwaryaLi}. Cheng and Zhou proved that
equality in the Bakry--\'Emery spectral-gap bound on a complete weighted
manifold forces a Gaussian factor \cite[Theorem~2]{ChengZhou}.
Quantitative Euclidean stability was developed by De~Philippis and
Figalli \cite[Section~4]{DePhilippisFigalli} and by Courtade and Fathi
\cite{CourtadeFathi}. Here equality at one interior time forces both
endpoint deficits to vanish and selects the same Gaussian direction at
the two endpoints.

The entropy application builds on the second-variation argument of
Aishwarya and Rotem
\cite[Lemma 4.3 and Theorem 4.6]{AishwaryaRotem}, formulated for even
measures and even velocity potentials. Aishwarya and Li
\cite[Theorems 1.5 and 1.7]{AishwaryaLi} obtained an entropy-power
inequality and a Gaussian Borell--Brascamp--Lieb inequality using a
coupling constructed from a common Gaussian source. Unlike that construction, our entropy
inequality is evaluated along the Wasserstein geodesic between the
prescribed endpoints. Bolley, Gentil, and Guillin derived dimensional
refinements of logarithmic Sobolev, Talagrand, and Brascamp--Lieb
inequalities \cite{BolleyGentilGuillin}. Eskenazis and Shenfeld developed
intrinsic dimensional functional inequalities on model spaces
\cite{EskenazisShenfeld}.

Cordero-Erausquin and Eskenazis proved a functional dimensional
Brunn--Minkowski inequality for rotationally invariant log-concave
measures \cite{CorderoEskenazis}. Malliaris, Melbourne, Roberto, and
Roysdon obtained Gaussian Borell--Brascamp--Lieb inequalities through
functional liftings of restricted geometric inequalities
\cite{MalliarisEtAl}. Our extension instead assumes
that the probability measures proportional to $f\,\d\gamma$ and
$g\,\d\gamma$ have the same barycenter.  It is proved along the
prescribed optimal transport interpolation. Kolesnikov and Milman
derived a weighted Brunn--Minkowski inequality from a boundary Poincar\'e
inequality along a geometric evolution
\cite{KolesnikovMilmanBoundary}. Eskenazis, Giannopoulos, and Tziotziou
proved a dimensional Brunn--Minkowski inequality for log-concave measures
by a diffusion argument \cite{EskenazisGiannopoulosTziotziou}.

\paragraph{Wasserstein barycenters.}
We next ask whether Theorem~\ref{thm:main} extends from two
endpoints to Wasserstein barycenters. Let
$\lambda_1,\ldots,\lambda_N>0$ with $\sum_{i=1}^N\lambda_i=1$, and let
$\bar\mu$ be a minimizer of
\[
\nu\longmapsto\sum_{i=1}^N\lambda_iW_2^2(\nu,\mu_i).
\]
For $N=2$, this reduces to a point on the Wasserstein geodesic, and
Theorem~\ref{thm:main} suggests the following question; see
\cite{AguehCarlier} for existence and basic properties of quadratic
Wasserstein barycenters.

\begin{question}[Wasserstein barycenter version]
\label{ques:barycenter-Poincare}
Suppose that $\mu_i$ is $\kappa_i$-strongly log-concave for
$i=1,\ldots,N$. Does their quadratic Wasserstein barycenter satisfy
\[
\sqrt{C_P(\bar\mu)}
\leq\sum_{i=1}^N\frac{\lambda_i}{\sqrt{\kappa_i}}\,?
\]
Can one also characterize equality by Gaussian factors in a common
direction, with variances $\kappa_i^{-1}$ for the respective input
measures?
\end{question}

For two endpoints, the Hessian of one Brenier potential and one scalar
Poisson equation produce a Bochner couple.  With more than two endpoints, one would need
compatible vector fields for several optimal transport maps and their
Hessians.  It is not clear how one Poisson equation could provide all
the required estimates. We leave
Question~\ref{ques:barycenter-Poincare} open. Han and Liu proved that the
Wasserstein Jensen inequality for Boltzmann entropy at barycenters of
arbitrary finite families characterizes finite-dimensional normed spaces
whose norm is induced by an inner product \cite{HanLiuHilbertian}.

\paragraph{Riemannian version.}
In a companion work in preparation, entitled \emph{Bochner Couples for
Optimal Transport on Riemannian Manifolds}, the authors study the
Riemannian analogue. There Jacobi fields describe the differential of
the transport maps in place of the Euclidean identity
$DT=\Hess{\Phi}$. A Codazzi-type tensor measures the failure of the matrix
symmetry used here. No result from that preprint is used in the
present paper.

\paragraph{Organization of the paper.}
Section~\ref{sec:smooth} constructs a Bochner couple from a weighted
Poisson equation and proves the interpolation estimate under smooth uniform
regularity assumptions.  Section~\ref{sec:approximation} develops a
curvature-preserving regularization and removes those auxiliary
assumptions by stability of optimal transport.  Section~\ref{sec:equality}
proves the equality characterization and identifies the maximal common
Gaussian factor of the endpoints.  Section~\ref{sec:Gaussian-applications}
applies the main theorem to centered Gaussian relative entropy, Gaussian
Brunn--Minkowski inequalities with barycenter terms, and Gaussian
functional inequalities after centering.

\section*{Acknowledgments}
We thank Heng Zhang, Shuhao Zhang, and Chiyu Zhou for helpful
discussions.

\section{Bochner couples from a weighted Poisson equation}
\label{sec:smooth}

We first prove the estimate under auxiliary uniform regularity
assumptions on the optimal transport map.  A weighted Poisson equation
involving the Hessian of the Brenier potential produces a Bochner couple.
Proposition~\ref{prop:smooth} interpolates its two fields.  Smoothness of
the endpoint potentials alone does not imply that the Brenier map is a
smooth global diffeomorphism.  The regularized potentials in
Section~\ref{sec:approximation} satisfy two-sided uniform Hessian bounds.
Caffarelli's contraction theory then gives global bi-Lipschitz control.
The regularity theorem of Cordero-Erausquin and Figalli for unbounded
supports gives smoothness.  The statement used in
Lemma~\ref{lem:transport-regularity} follows directly
from \cite[Theorem 2.2]{KolesnikovContractions} and
\cite[Remark 4 and Corollary 5]{CorderoFigalliUnbounded}.

Throughout this section, assume that $V_0,V_1$ are smooth and that
$\d\mu_i=e^{-V_i}\,\d x$ with $\Hess{V_i}\succeq\kappa_iI$.  Assume also
that the Brenier map $T=\nabla\Phi$ from $\mu_0$ to $\mu_1$ is a smooth
diffeomorphism.  Finally, assume that $H=\Hess{\Phi}$ and $H^{-1}$ are
bounded in global operator norm.  Equivalently,
$\lambda I\preceq H\preceq\Lambda I$ for some
$0<\lambda\leq\Lambda<\infty$.  These are temporary hypotheses for the
smooth argument and will be obtained from the regularization in
Section~\ref{sec:approximation}.
For symmetric matrices, $A\preceq B$ means that $B-A$ is positive
semidefinite, and $\norm{A}_{\HS}$ denotes the Hilbert--Schmidt norm.

For $\d\mu=e^{-V}\,\d x$, the weighted divergence with respect to
$\mu$ is defined through integration by parts and is given by 
\begin{equation}\label{eq:weighted-divergence-definition}
 \diver_\mu X=\diver X-\ip{\nabla V}{X}.
\end{equation}
We use the convention $(DX)_{ij}=\partial_jX_i$.

\begin{lemma}[Weighted vector Bochner identity]\label{lem:vector-bochner}
For every $X\in C_c^\infty(\R^n;\R^n)$,
\begin{equation}\label{eq:vector-bochner}
 \int(\diver_\mu X)^2\,\d\mu
 =
 \int\left\{\tr((DX)^2)+\Hess{V}[X,X]\right\}\,\d\mu.
\end{equation}
\end{lemma}

\begin{proof}
The weighted integration by parts yields 
\begin{equation}
	 \int(\diver_\mu X)^2\,\d\mu
	=-\int X_j\partial_j(\diver_\mu X)\,\d\mu,
\end{equation}
where repeated indices are summed.  Expanding the derivative then implies
\begin{equation}\label{eq:vector-bochner-expansion}
	 \begin{split}
		\int(\diver_\mu X)^2\,\d\mu
		&=-\int X_j\partial_i\partial_jX_i\,\d\mu
		+\int V_{ij}X_iX_j\,\d\mu
		+\int V_iX_j\partial_jX_i\,\d\mu.
	\end{split}
\end{equation}
Integrating the first term by parts in the $i$-th coordinate implies 
\begin{equation}\label{eq:vector-bochner-integration-by-parts}
	-\int X_j\partial_i(\partial_jX_i)\,\d\mu
	=\int(\partial_iX_j-V_iX_j)\partial_jX_i\,\d\mu.
\end{equation}
Combining \eqref{eq:vector-bochner-expansion} and
\eqref{eq:vector-bochner-integration-by-parts} then yields 
\eqref{eq:vector-bochner}.
\end{proof}

\begin{remark}
	The first term in \eqref{eq:vector-bochner} is not, in general, the
	Hilbert--Schmidt norm of $DX$:
	\[
	\tr((DX)^2)
	=\sum_{i,j}\partial_jX_i\,\partial_iX_j,
	\qquad
	\norm{DX}_{\HS}^2
	=\sum_{i,j}(\partial_jX_i)^2.
	\]
	The two expressions agree when $DX$ is symmetric, but
	$\tr((DX)^2)$ can have either sign for a general vector field.  Thus
	\eqref{eq:vector-bochner} cannot be applied directly at an intermediate
	time: the potential of $\mu_t$ need not have a useful Hessian lower
	bound, and the trace term has no sign for a general vector field with
	the prescribed divergence.
		Proposition~\ref{prop:poisson} instead works at the endpoints, where the curvature
		bounds are available, and uses the symmetry of $H DX$, with
		$H=\Hess{\Phi}$, to recover
	the nonnegativity of the trace terms.
\end{remark}

\begin{definition}[Smooth Bochner couple]
\label{def:bochner-couple}
Let $T=\nabla\Phi$ transport $\mu_0$ to $\mu_1$ and set
$H=\Hess{\Phi}$. Given a smooth mean-zero datum $g$ and constants
$a_0,a_1\geq0$, a pair of smooth
vector fields $(X,Y)$ is called a \emph{Bochner couple} with datum $g$
and constants $(a_0,a_1)$ if the relations
\begin{equation}\label{eq:bochner-couple}
 Y\circ T=HX,
 \qquad
 -\diver_{\mu_0}X=g,
 \qquad
 -(\diver_{\mu_1}Y)\circ T=g,
\end{equation}
hold, the two matrix symmetry conditions
\begin{equation}\label{eq:bochner-couple-symmetry}
 H DX\ \text{is symmetric},
 \qquad
 ((DY)\circ T)H\ \text{is symmetric}
\end{equation}
hold, and the endpoint Bochner estimates
\begin{equation}\label{eq:bochner-couple-bounds}
 \norm{X}_{L^2(\mu_0)}\leq a_0\norm{g}_{L^2(\mu_0)},
 \qquad
 \norm{Y}_{L^2(\mu_1)}\leq a_1\norm{g}_{L^2(\mu_0)}
\end{equation}
are valid.  If there is a scalar function $u$ such that
\begin{equation}\label{eq:poisson-generated-couple}
 X=H^{-1}\nabla u,
 \qquad
 Y\circ T=\nabla u
\end{equation}
then we say that the couple is generated by the weighted Poisson
equation. Here $H=DT=\Hess{\Phi}$ is the differential of the Brenier map
that relates the two fields.
\end{definition}

We call constructing a Bochner couple from common data and then
interpolating its two fields \emph{Bochner coupling}. 

\begin{proposition}[Endpoint estimates for the weighted Poisson equation]
\label{prop:poisson}
For every $g\in L^2(\mu_0)$ with $\int g\,\d\mu_0=0$, there exists a unique
$u\in H^1(\mu_0)$ satisfying $\int u\,\d\mu_0=0$ and 
\begin{equation}\label{eq:poisson}
 -\diver_{\mu_0}(H^{-1}\nabla u)=g
\end{equation}
in the weak sense. Setting $X=H^{-1}\nabla u$ and defining $Y$ by
$Y\circ T=\nabla u$, one also has
\begin{equation}
	-(\diver_{\mu_1}Y)\circ T=g 
\end{equation}
in the weak sense, together with the estimates 
\begin{equation}\label{eq:endpoint-bounds}
\norm{X}_{L^2(\mu_0)}
 \leq\kappa_0^{-1/2}\norm{g}_{L^2(\mu_0)},\qquad
 \norm{Y}_{L^2(\mu_1)}=\norm{\nabla u}_{L^2(\mu_0)}
 \leq\kappa_1^{-1/2}\norm{g}_{L^2(\mu_0)}.
\end{equation}
\end{proposition}

\begin{proof}
\medskip\noindent\textbf{Step 1: Endpoint identities and Bochner estimates.}
We begin with $u\in C_c^\infty(\R^n)$.  Set  
\begin{equation*}
 X=H^{-1}\nabla u,\qquad
 Bu=-\diver_{\mu_0}X,\qquad
 Y(Tx)=\nabla u(x)=H(x)X(x).
\end{equation*}
Both fields are compactly supported, and we will verify that $(X,Y)$ is
a Bochner couple with datum $Bu$ and constants
$(\kappa_0^{-1/2},\kappa_1^{-1/2})$.
We first derive the endpoint estimates (the a priori estimates) with
$g=Bu$; the extension to an arbitrary
mean-zero $g\in L^2(\mu_0)$ is carried out at the end of the proof by a
density argument. For every $\psi\in C_c^\infty(\R^n)$,
\begin{equation}\label{eq:source-target-pairing}
 \int\ip{\nabla\psi}{Y}\,\d\mu_1=\int \ip{\nabla\psi(Tx)}{H(x)X(x)}\,\d\mu_0(x)
 =\int\ip{\nabla(\psi\circ T)}{X}\,\d\mu_0, 
\end{equation}
which implies 
\begin{equation}\label{eq:source-target-divergence}
 (\diver_{\mu_1}Y)\circ T=\diver_{\mu_0}X=-Bu
\end{equation}
in the weak sense.  This transformation rule for weighted divergence
under the Brenier map appears in
\cite[Lemma~2.1 and (2.5)]{KolesnikovHessian}.
Let $P=DX$, so that $P_{kj}=\partial_jX_k$. Since
$H_{ik}=(\Hess{\Phi})_{ik}=\partial_i\partial_k\Phi$,
differentiating $(HX)_i=\partial_i u$ with respect to $x_j$ yields 
\begin{equation}
	(\Hess{u})_{ij}
	=(\partial_jH_{ik})X_k+H_{ik}P_{kj}
	=\Phi_{ikj}X_k+H_{ik}P_{kj}, 
\end{equation}
and hence
\begin{equation}\label{eq:source-Hessian-decomposition}
 \Hess{u}=C_X+HP,\qquad (C_X)_{ij}=\Phi_{ijk}X_k.
\end{equation}

The smoothness of $\Phi$ implies that $C_X$ is symmetric. Since
$\Hess{u}$ is symmetric as well,
\eqref{eq:source-Hessian-decomposition} shows that
$HP=\Hess{u}-C_X$ is symmetric. The matrix
\[
 H^{-1/2}(HP)H^{-1/2}=H^{1/2}PH^{-1/2}
\]
is symmetric and similar to $P$. Hence,
\begin{equation}\label{eq:source-trace-positive}
 \tr(P^2)
 =\norm{H^{-1/2}(HP)H^{-1/2}}_{\HS}^2\geq0.
\end{equation}
Applying Lemma~\ref{lem:vector-bochner} with $\mu_0$ and using
\eqref{eq:source-trace-positive}, we obtain 
\begin{equation}\label{eq:source-bound-smooth}
\norm{X}_{L^2(\mu_0)}=\norm{H^{-1}\nabla u}_{L^2(\mu_0)}
 \leq\kappa_0^{-1/2}\norm{Bu}_{L^2(\mu_0)}=\kappa_0^{-1/2}\norm{g}_{L^2(\mu_0)}. 
\end{equation}

At the target, let $Q(x)=(D_yY)(Tx)$. Since 
$Y(Tx)=\nabla u(x)$, we have $QH=\Hess{u}$. Thus $Q$ is similar to the
symmetric matrix $H^{-1/2}\Hess{u}H^{-1/2}$, and hence
\begin{equation}\label{eq:target-trace-positive}
 \tr(Q^2)
 =\norm{H^{-1/2}\Hess{u}H^{-1/2}}_{\HS}^2\geq0.
\end{equation}
Applying Lemma~\ref{lem:vector-bochner} with $\mu_1$, together with
\eqref{eq:source-target-divergence} and
\eqref{eq:target-trace-positive}, we obtain 
\begin{equation}\label{eq:target-bound-smooth}
  \norm{Y}_{L^2(\mu_1)}=\norm{\nabla u}_{L^2(\mu_0)}
 \leq\kappa_1^{-1/2}\norm{Bu}_{L^2(\mu_0)}.
\end{equation}
Thus the two bounds are the endpoint estimates of a Bochner couple with constants
$(\kappa_0^{-1/2},\kappa_1^{-1/2})$.

\medskip\noindent\textbf{Step 2: Weak solvability and range density.}
The existence and uniqueness of a centered weak solution of
\eqref{eq:poisson} already follow from the Lax--Milgram theorem
\cite[Corollary 5.8]{BrezisFunctional}, applied on the mean-zero
subspace of $H^1(\mu_0)$ to the bilinear form
\[
 (v,w)\longmapsto
 \int\ip{H^{-1}\nabla v}{\nabla w}\,\d\mu_0.
\]
Here boundedness and coercivity follow from the global bounds on $H$
and $H^{-1}$ and the Poincar\'e inequality for $\mu_0$.  To ensure that
this solution also inherits \eqref{eq:source-bound-smooth} and
\eqref{eq:target-bound-smooth}, it remains to prove that the range of
the operator below is dense.

Let $B=-\diver_{\mu_0}(H^{-1}\nabla)$ on
$C_c^\infty(\R^n)$.  Its range is dense in
\[
 L_0^2(\mu_0)
 =\left\{h\in L^2(\mu_0):\int h\,\d\mu_0=0\right\}.
\]
Indeed, if $h\in L_0^2(\mu_0)$ is orthogonal to the range and
$a=e^{-V_0}H^{-1}$, then
\begin{equation}
	 -\int h\diver(a\nabla\varphi)\,\d x=0,
	\qquad \varphi\in C_c^\infty(\R^n). 
\end{equation}
The operator $-\diver(a\nabla)$ is smooth and locally uniformly
elliptic, so standard interior elliptic regularity
\cite{HormanderInterior} yields $h\in H^1_{\mathrm{loc}}$ (indeed,
$h\in C^\infty$). Choosing a test function $\eta_R^2h$, where
$\eta_R=1$ on $B_R$, $\operatorname{supp}\eta_R\subseteq B_{2R}$, and
$|\nabla\eta_R|\leq C/R$, gives
\begin{equation}
	\begin{aligned}
		 \int\eta_R^2
		\ip{H^{-1}\nabla h}{\nabla h}\,\d\mu_0&=-2 \int\eta_Rh
		\ip{H^{-1}\nabla h}{\nabla\eta_R}\,\d\mu_0\\
		&\leq
		\frac12\int\eta_R^2
		\ip{H^{-1}\nabla h}{\nabla h}\,\d\mu_0
		+2\int h^2
		\ip{H^{-1}\nabla\eta_R}{\nabla\eta_R}\,\d\mu_0. 
	\end{aligned}
\end{equation}
Combining with $H^{-1}\preceq\lambda^{-1}I$, we obtain 
\begin{equation}
	\int\eta_R^2
	\ip{H^{-1}\nabla h}{\nabla h}\,\d\mu_0
	\leq
	\frac{4C^2}{\lambda R^2}
	\norm{h}_{L^2(\mu_0)}^2. 
\end{equation}
Letting $R\to\infty$ on each fixed ball yields $\nabla h=0$; the
zero-mean condition then implies $h=0$. Thus the range is dense.

\medskip\noindent\textbf{Step 3: Approximation and passage to the limit.}
Choose $u_k^0\in C_c^\infty(\R^n)$ such that
$Bu_k^0\to g$ in $L^2(\mu_0)$, and set
$u_k=u_k^0-\int u_k^0\,\d\mu_0$.
Then $u_k$ is centered, while $\nabla u_k=\nabla u_k^0$ and 
$Bu_k=Bu_k^0$. Applying \eqref{eq:source-bound-smooth} and
\eqref{eq:target-bound-smooth} to
$u_k^0-u_\ell^0\in C_c^\infty(\R^n)$ shows that
$\nabla u_k$ and $H^{-1}\nabla u_k$ are Cauchy in
$L^2(\mu_0)$. Moreover, since each $u_k$ is centered, the endpoint Poincar\'e inequality and
the Cauchy property of $(\nabla u_k)$ show that $(u_k)$ is Cauchy in
$H^1(\mu_0)$. It follows that $u_k\to u$ in $H^1(\mu_0)$ for some centered
$u\in H^1(\mu_0)$.
For every $\varphi\in C_c^\infty(\R^n)$, 
\begin{equation}
	\begin{aligned}
		\int\ip{H^{-1}\nabla u}{\nabla\varphi}\,\d\mu_0=
		\lim_{k\to\infty}
		\int\ip{H^{-1}\nabla u_k}{\nabla\varphi}\,\d\mu_0=
		\lim_{k\to\infty}
		\int(Bu_k)\varphi\,\d\mu_0
		=
		\int g\varphi\,\d\mu_0.
	\end{aligned}
\end{equation}
Thus $u$ is a centered weak solution of \eqref{eq:poisson}. By the density of $C_c^\infty(\R^n)$ in $H^1(\mu_0)$, the weak
identity extends to all $H^1(\mu_0)$ test functions. Hence $u$
coincides with the unique solution provided by the Lax--Milgram theorem. Define $Y_k$ by
$Y_k\circ T=\nabla u_k$. 
Since $u_k\to u$ in $H^1(\mu_0)$, we have
$Y_k\to Y$ in $L^2(\mu_1)$. Applying
\eqref{eq:source-target-pairing} to $u_k^0$ and then passing to the limit, we obtain
\begin{equation}
\int\ip{\nabla\psi}{Y}\,\d\mu_1
=
\int g\,\psi\circ T\,\d\mu_0, 
	\qquad \psi\in C_c^\infty(\R^n), 
\end{equation}
which implies $-(\diver_{\mu_1}Y)\circ T=g$ in the weak sense.  Passing to
the limit in \eqref{eq:source-bound-smooth} and
\eqref{eq:target-bound-smooth} yields \eqref{eq:endpoint-bounds}.
\end{proof}

At the smooth level, the same Poisson solution $u$ gives both matrix
identities.  The matrix $HDX$ is symmetric at the source, and
$((D_yY)\circ T)H=\Hess{u}$ is symmetric at the target.  These facts
give the sharp endpoint estimates.  They do not require a curvature
bound for any intermediate measure.  Together with
\eqref{eq:source-target-divergence}, these estimates provide the input
for the affine interpolation in Proposition~\ref{prop:smooth}.

\begin{proposition}[Poincar\'e estimate in the smooth case]
	\label{prop:smooth}
	Assume that $V_i\in C^\infty(\R^n)$ and
	$\Hess{V_i}\succeq\kappa_iI$ for $i=0,1$. Suppose that the Brenier map
	$T=\nabla\Phi$ from $\mu_0$ to $\mu_1$ is a smooth diffeomorphism and
	that $\lambda I\preceq\Hess{\Phi}\preceq\Lambda I$ for some $0<\lambda\leq\Lambda<\infty$. Then \eqref{eq:main} holds for
	every $t\in[0,1]$.
\end{proposition}

\begin{proof}
Fix $f\in C_c^\infty(\R^n)$ and let $h=f-\int f\,\d\mu_t$.  Set
\begin{equation*}
 A_t=DF_t=(1-t)I+tH,\qquad
 \alpha_t=(1-t)+t\lambda,\qquad
 g=h\circ F_t.
\end{equation*}
Since $(F_t)_\#\mu_0=\mu_t$, we have
$g\in L_0^2(\mu_0)$ and
$\norm{g}_{L^2(\mu_0)}=\norm{h}_{L^2(\mu_t)}$.
 The uniform lower bound
$A_t\succeq\alpha_tI$ yields 
\begin{equation}\label{eq:strong-monotonicity-Ft}
	 \bigl(F_t(x)-F_t(y)\bigr)\cdot(x-y)
	\geq\alpha_t|x-y|^2,
\end{equation}
which implies that $F_t$ is injective and proper. The local inverse
function theorem makes its image open, while properness makes the image
closed; hence $F_t$ is a global diffeomorphism.

Apply Proposition~\ref{prop:poisson} to
$g=h\circ F_t$ and denote $X=H^{-1}\nabla u$. Define $q$ at time $t$ by
\begin{equation*}
 q(F_t(x))=A_t(x)X(x)
       =(1-t)X(x)+t\nabla u(x).
\end{equation*}

For $\psi\in C_c^\infty(\R^n)$, properness of $F_t$ ensures that
$\psi\circ F_t$ is compactly supported.  Using
$(F_t)_\#\mu_0=\mu_t$, the definition $q\circ F_t=DF_t\,X$, and the
chain rule, we obtain 
\begin{equation}\label{eq:intermediate-divergence-identity}
	\begin{aligned}
		 \int q\cdot\nabla\psi\,\d\mu_t
		&=\int q(F_t(x))\cdot\nabla\psi(F_t(x))\,\d\mu_0(x)\\
		&=\int X\cdot\nabla(\psi\circ F_t)\,\d\mu_0\\
		&\overset{*}{=}\int g\,\psi\circ F_t\,\d\mu_0
		=\int h\psi\,\d\mu_t, 
	\end{aligned}
\end{equation}
where $(*)$ follows from the weak formulation of
$-\diver_{\mu_0}X=g$. Hence
$-\diver_{\mu_t}q=h$ in the weak sense. Moreover,
Proposition~\ref{prop:poisson} gives
\begin{equation}\label{eq:interpolated-lift-bound}
	 \norm{q}_{L^2(\mu_t)}
	\leq(1-t)\norm{X}_{L^2(\mu_0)}
	+t\norm{\nabla u}_{L^2(\mu_0)}
	\leq\left(\frac{1-t}{\sqrt{\kappa_0}}
	+\frac{t}{\sqrt{\kappa_1}}\right)
	\norm{h}_{L^2(\mu_t)}.
\end{equation}

Since $h=f-\int f\,\d\mu_t$, taking $\psi=f$ in \eqref{eq:intermediate-divergence-identity}, we obtain 
\begin{equation}\label{eq:smooth-poincare-estimate}
 \Var_{\mu_t}(f)
 =\int q\cdot\nabla f\,\d\mu_t
 \leq\left(\frac{1-t}{\sqrt{\kappa_0}}
            +\frac{t}{\sqrt{\kappa_1}}\right)
       \sqrt{\Var_{\mu_t}(f)}
       \left(\int|\nabla f|^2\,\d\mu_t\right)^{1/2},
\end{equation}
which proves the proposition.
\end{proof}

\section{Regularization and transport stability}\label{sec:approximation}
To pass to nonsmooth endpoint potentials, we need a smooth
approximation that preserves their curvature lower bounds.  Gaussian
convolution alone weakens these bounds, so we combine it with a
compensating dilation.  The covariance estimate needed to verify this
regularization is the following linear form of the Brascamp--Lieb
inequality.

\begin{lemma}\label{lem:linear-BL}
Let $U:\R^n\to(-\infty,+\infty]$ be proper and lower semicontinuous,
with $U-\lambda|x|^2/2$ convex for some $\lambda>0$.  If
$0<\int e^{-U}\,\d x<\infty$, $\d\nu\propto e^{-U}\,\d x$, and $\zeta$ has law
$\nu$, then
\begin{equation}
	 \Cov(\zeta)\preceq\lambda^{-1}I. 
\end{equation}
\end{lemma}

\begin{proof}
Write $U=\lambda|x|^2/2+W$, where $W$ is proper, lower semicontinuous,
and convex. Choose an affine minorant $W(x)\geq p\cdot x+b$, and let
\[
 W_m(x)=\inf_{y\in\R^n}\left\{W(y)+\frac{m}{2}|x-y|^2\right\},
 \qquad m\geq1,
\]
be the Moreau envelopes of $W$. They are finite and convex,
$W_m\uparrow W$ pointwise, and
\begin{equation}\label{eq:Moreau-affine-minorant}
 W_m(x)\geq p\cdot x+b-\frac{|p|^2}{2m}.
\end{equation}
Let $\rho$ be a smooth, compactly supported, radial probability density,
write $\rho_\delta(x)=\delta^{-n}\rho(x/\delta)$, and set
$W_{m,\delta}=W_m*\rho_\delta$. Then $W_{m,\delta}$ is smooth and
convex. Because $\int z\rho_\delta(z)\,\d z=0$,
\eqref{eq:Moreau-affine-minorant} also holds with $W_{m,\delta}$ in
place of $W_m$.

The classical Brascamp--Lieb inequality \cite{BrascampLieb}, applied to
the smooth potential $\lambda|x|^2/2+W_{m,\delta}$, gives
\begin{equation}\label{eq:smooth-linear-BL}
 \Var_{\nu_{m,\delta}}(a\cdot x)
 \leq\lambda^{-1}|a|^2,
 \qquad a\in\R^n,
\end{equation}
where $\d\nu_{m,\delta}\propto
e^{-\lambda|x|^2/2-W_{m,\delta}(x)}\,\d x$.
For fixed $m$, the finite convex function $W_m$ is continuous and
$W_{m,\delta}\to W_m$ locally uniformly as $\delta\downarrow0$.
Moreover, \eqref{eq:Moreau-affine-minorant}, uniformly for $m\geq1$ and
$\delta>0$, bounds all the corresponding unnormalized densities by a
constant times $e^{-\lambda|x|^2/2-p\cdot x}$. This function remains
integrable after multiplication by $1+|x|^2$, so dominated convergence
applies to the normalization constants and to the first and second
moments. We may first let $\delta\downarrow0$ in
\eqref{eq:smooth-linear-BL} and then let $m\to\infty$. Since
$W_m\uparrow W$, this yields
\[
 \Var_\nu(a\cdot x)\leq\lambda^{-1}|a|^2,
 \qquad a\in\R^n,
\]
which is equivalent to the stated covariance bound.
\end{proof}

\begin{lemma}\label{lem:regularization}
Let $\kappa>0$, and let $\mu$ have potential $V$ with
$V-\kappa|x|^2/2$ convex. Let $\xi\sim\mu$ and
$G\sim N(0,I)$ be independent.  For $\varepsilon>0$, define
\begin{equation*}
 \xi^\varepsilon
 =\frac{\xi+\sqrt{\varepsilon}\,G}
        {\sqrt{1+\kappa\varepsilon}},
 \qquad \mu^\varepsilon=\operatorname{Law}(\xi^\varepsilon).
\end{equation*}
Then $\mu^\varepsilon\to\mu$ in $W_2$.  The measure
$\mu^\varepsilon$ has a smooth positive density
$e^{-V^\varepsilon}$ whose potential satisfies
\begin{equation}\label{eq:regularized-Hessian}
 \kappa I\preceq \Hess{V^\varepsilon}
 \preceq\frac{1+\kappa\varepsilon}{\varepsilon}I.
\end{equation}
\end{lemma}

The lower Hessian bound is also the classical convolution rule for
strong log-concavity; see \cite[Proposition~5.5]{SaumardWellner}.  We
include the conditional-moment proof because it simultaneously gives 
the upper Hessian bound and covers extended-valued endpoint potentials. 

\begin{proof}
\medskip\noindent\textbf{Step 1: Convergence in $W_2$.}
Set $a_\varepsilon=\sqrt{1+\kappa\varepsilon}$. Writing $V=\kappa|x|^2/2+W$ with $W$ convex, choose an affine
minorant $W(x)\geq p\cdot x+b$. Then $e^{-V(x)}$ is bounded by a
constant times $\exp(-\kappa|x|^2/2-p\cdot x)$, and hence
$\mathbb E|\xi|^2<\infty$. Couple $\mu^\varepsilon$ and $\mu$ by using
the same pair $(\xi,G)$. By the definition of $W_2$,
\begin{equation} \label{eq:regularization-W2}
	\begin{aligned}
		W_2^2(\mu^\varepsilon,\mu)
		&\leq \mathbb E\left|
		(a_\varepsilon^{-1}-1)\xi
		+\frac{\sqrt\varepsilon}{a_\varepsilon}G
		\right|^2 \\
		&=(a_\varepsilon^{-1}-1)^2\mathbb E|\xi|^2
		+\frac{\varepsilon}{a_\varepsilon^2}\mathbb E|G|^2 =(a_\varepsilon^{-1}-1)^2\mathbb E|\xi|^2
		+\frac{\varepsilon n}{a_\varepsilon^2}\longrightarrow0, 
	\end{aligned}
\end{equation}
where the cross term vanishes since $G$ is independent of $\xi$ and
$\mathbb E[G]=0$. This proves $\mu^\varepsilon\to\mu$ in $W_2$.

\medskip\noindent\textbf{Step 2: Hessian bounds after Gaussian smoothing.}
Set $Y_\varepsilon=\xi+\sqrt{\varepsilon}G$, let
$\phi_\varepsilon$ be the density of $N(0,\varepsilon I)$, and write
\begin{equation}\label{eq:gaussian-convolution-density}
	p_\varepsilon(y)
	=\int\phi_\varepsilon(y-x)\,\d\mu(x)
	=e^{-U_\varepsilon(y)}.
\end{equation}
This density is smooth and strictly positive. For each $y\in\R^n$,
Bayes' rule gives
\begin{equation}
	\mathbb P(\xi\in\d x\mid Y_\varepsilon=y)
	=\frac{\phi_\varepsilon(y-x)}{p_\varepsilon(y)}\,\d\mu(x). 
\end{equation}
Define $m_\varepsilon(y)=\mathbb E[\xi\mid Y_\varepsilon=y]$. Then, for
$j=1,\ldots,n$,
\begin{equation}\label{eq:gaussian-convolution-density0}
	p_\varepsilon(y)m_{\varepsilon,j}(y)
	=\int x_j\phi_\varepsilon(y-x)\,\d\mu(x). 
\end{equation}

For $1\leq j,k\leq n$, differentiation of
\eqref{eq:gaussian-convolution-density} and \eqref{eq:gaussian-convolution-density0} yields 
\begin{equation}
	\partial_k p_\varepsilon(y)
	=-\varepsilon^{-1}
	\int(y_k-x_k)\phi_\varepsilon(y-x)\,\d\mu(x), 
\end{equation}
and 
\begin{equation}
	\partial_k\!\left(p_\varepsilon(y)m_{\varepsilon,j}(y)\right)
	=-\varepsilon^{-1}
	\int x_j(y_k-x_k)\phi_\varepsilon(y-x)\,\d\mu(x). 
\end{equation} 
The quotient rule therefore yields
\begin{equation}\label{eq:score-conditional-mean}
	\nabla\log p_\varepsilon(y)
	=-\varepsilon^{-1}\bigl(y-m_\varepsilon(y)\bigr),
	\qquad
	Dm_\varepsilon(y)
	=\varepsilon^{-1}\Cov(\xi\mid Y_\varepsilon=y).
\end{equation}
Indeed, the $(j,k)$ entry in the second identity is $\varepsilon^{-1}
\left(
\mathbb E[\xi_j\xi_k\mid Y_\varepsilon=y]
-m_{\varepsilon,j}(y)m_{\varepsilon,k}(y)
\right)$. 
Since $\mathbb E|\xi|^2<\infty$, differentiation under the integral
sign is justified by dominated convergence, locally uniformly in $y$. 
The first identity in \eqref{eq:score-conditional-mean} is the Gaussian
Tweedie formula.  The second expresses the derivative of the conditional
mean in terms of the conditional covariance.  See Efron \cite{EfronTweedie} for the
statistical formulation and Dytso, Poor, and Shamai
\cite{DytsoPoorShamai} for vector Gaussian derivative identities.
Since $U_\varepsilon=-\log p_\varepsilon$, differentiating the first
identity in \eqref{eq:score-conditional-mean} and substituting the
formula for $Dm_\varepsilon$ yields
\begin{equation}\label{eq:conditional-covariance}
 \Hess{U_\varepsilon}(y)
 =\varepsilon^{-1}I
 -\varepsilon^{-2}
   \Cov(\xi\mid Y_\varepsilon=y).
\end{equation}
Thus, up to an additive normalizing constant, the conditional law above has potential
\[
 x\longmapsto V(x)+\frac{|y-x|^2}{2\varepsilon},
\]
which is $(\kappa+\varepsilon^{-1})$-strongly convex in the
extended-valued sense.  Therefore, by Lemma~\ref{lem:linear-BL},
\begin{equation}\label{eq:posterior-covariance-bound}
		\Cov(\xi\mid Y_\varepsilon=y)
		\preceq
		\frac{\varepsilon}{1+\kappa\varepsilon}I.
\end{equation}
Substituting \eqref{eq:posterior-covariance-bound} into
\eqref{eq:conditional-covariance} gives
the lower Hessian bound, while
$\Cov(\xi\mid Y_\varepsilon=y)\succeq0$ gives the upper bound:
\begin{equation}
		\frac{\kappa}{1+\kappa\varepsilon}I
		\preceq \Hess{U_\varepsilon}
	\preceq\varepsilon^{-1}I.
\end{equation}

Finally, since $\xi^\varepsilon=Y_\varepsilon/a_\varepsilon$ and
$a_\varepsilon^2=1+\kappa\varepsilon$, the change of variables yields
\begin{equation}
	V^\varepsilon(x)
	=U_\varepsilon(a_\varepsilon x)-n\log a_\varepsilon,
	\qquad
	\Hess{V^\varepsilon}(x)
	=a_\varepsilon^2\Hess{U_\varepsilon}(a_\varepsilon x), 
\end{equation}
which proves \eqref{eq:regularized-Hessian}.
\end{proof}

\begin{lemma}[Regularity and bounds for the Brenier map]
\label{lem:transport-regularity}
Let $\d\mu_i=e^{-V_i}\,\d x$, $i=0,1$, be smooth positive
probability densities whose potentials satisfy
\[
 \lambda_i^- I\preceq \Hess{V_i}\preceq\lambda_i^+ I,
 \qquad 0<\lambda_i^-\leq\lambda_i^+<\infty. 
\]
Then the Brenier map $T=\nabla\Phi$ from $\mu_0$ to $\mu_1$ is a smooth
global diffeomorphism and
\begin{equation}\label{eq:caffarelli-bounds}
 \sqrt{\frac{\lambda_0^-}{\lambda_1^+}}\,I
 \preceq \Hess{\Phi}
 \preceq\sqrt{\frac{\lambda_0^+}{\lambda_1^-}}\,I.
\end{equation}
\end{lemma}

\begin{proof}
	Since the endpoint densities are smooth and positive on $\R^n$, they
and their reciprocals are bounded on every compact set.  Hence
\cite[Corollary 5]{CorderoFigalliUnbounded}, applied for every order of
regularity, shows that $T=\nabla\Phi$ is a smooth global diffeomorphism.
Moreover, its inverse is the reverse Brenier map
$S:\mu_1\to\mu_0$; see
\cite[Remark 4]{CorderoFigalliUnbounded}.

Caffarelli's contraction theorem \cite{Caffarelli}, in the
two-potential form of \cite[Theorem~2.2]{KolesnikovContractions},
applied first to $T$ and then to $S$, yields 
\begin{equation}
	\operatorname{Lip}(T)
	\leq\sqrt{\frac{\lambda_0^+}{\lambda_1^-}},
	\qquad
	\operatorname{Lip}(S)
	\leq\sqrt{\frac{\lambda_1^+}{\lambda_0^-}}. 
\end{equation}
Since $S=T^{-1}$ and both maps are smooth, we have
$DS(T(x))=DT(x)^{-1}$.
The preceding Lipschitz bounds therefore imply
\begin{equation}
	\sqrt{\frac{\lambda_0^-}{\lambda_1^+}}\,I
	\preceq DT=\Hess{\Phi}
	\preceq\sqrt{\frac{\lambda_0^+}{\lambda_1^-}}\,I, 
\end{equation}
which proves \eqref{eq:caffarelli-bounds}.
\end{proof}

\begin{proof}[Proof of Theorem~\ref{thm:main}]
Apply Lemma~\ref{lem:regularization} to both endpoints with their respective
curvature constants.  Let $\varepsilon_j\downarrow0$, denote the
regularized endpoints by $\mu_i^j$, $i=0,1$, and let $\pi_j$ be their
quadratic optimal plans. By Lemma~\ref{lem:regularization} and
Lemma~\ref{lem:transport-regularity}, the assumptions of
Proposition~\ref{prop:smooth} hold with the same lower curvature constants
$\kappa_i$.

By the stability theorem for optimal transport plans
\cite[Theorem~5.20]{Villani}, every weak subsequential limit of
$(\pi_j)$ is optimal between $\mu_0$ and $\mu_1$.  Brenier's theorem
\cite{Brenier} gives a unique quadratic optimal plan $\pi$ because
$\mu_0\ll\mathcal L^n$.  Hence $\pi_j\to\pi$ weakly.  The marginals
converge in $W_2$, so the joint second moments also converge.  It follows
from \cite[Theorem~6.9]{Villani} that
$\pi_j\to\pi$ in $W_2(\R^n\times\R^n)$.  Fix $t\in[0,1]$ and let
$L_t(x,y)=(1-t)x+ty$. Since $L_t$ is Lipschitz,
\begin{equation}
	 \mu_t^j:=(L_t)_\#\pi_j
	\rightarrow (L_t)_\#\pi=\mu_t
	\quad\text{in }W_2. 
\end{equation}
For $f\in C_c^\infty(\R^n)$, by Proposition~\ref{prop:smooth}, we obtain 
\begin{equation}
	\Var_{\mu_t^j}(f)
	\leq\left(\frac{1-t}{\sqrt{\kappa_0}}
	+\frac{t}{\sqrt{\kappa_1}}\right)^2
	\int|\nabla f|^2\,\d\mu_t^j.
\end{equation}
Since $f$, $f^2$, and $|\nabla f|^2$ are bounded and continuous,
passing to the limit proves \eqref{eq:main}. Since $t\in[0,1]$ was
arbitrary, the conclusion holds along the entire geodesic.
\end{proof}

\section{Gaussian rigidity}\label{sec:equality}
Theorem~\ref{thm:equality} states that equality at an interior time is
equivalent to a common Gaussian factor.  The variances of the two
Gaussian factors are $\kappa_i^{-1}$.  We first prove a quantitative
rigidity lemma at one endpoint.  It shows that a near-extremizer of the
Poincar\'e inequality has an almost constant gradient.  We apply this
lemma to regularized near-extremizers for $\mu_t$.  The resulting limits
are affine in the same direction at both endpoints, which gives the
common Gaussian factorization.  The converse follows by splitting the
optimal transport along this Gaussian direction.
Quantitative affine approximation of Poincar\'e near-extremizers for
strongly log-concave measures was developed by De~Philippis and Figalli
\cite[Section~4]{DePhilippisFigalli}. Courtade and Fathi proved an
$O(\varepsilon)$ squared $L^2$-gradient estimate for normalized
near-extremizers \cite[Lemma~2.2]{CourtadeFathi}. The next lemma is its
rescaled, homogeneous one-function form; we include the short derivation
to make the scaling explicit.

\begin{lemma}[Homogeneous form of quantitative Poincar\'e stability]\label{lem:near-equality}
Let $\nu$ be the probability measure $\d\nu=e^{-W}\,\d x$, where $W$ is
smooth and
\[
 \Hess{W}\succeq\kappa I,\qquad \kappa>0.
\]
If $h\in H^1(\nu)$ satisfies $\int h\,\d\nu=0$, set
\[
 D=\int|\nabla h|^2\,\d\nu-\kappa\int h^2\,\d\nu,\qquad
 a=\int\nabla h\,\d\nu.
\]
Then $D\geq0$ and
\begin{equation}\label{eq:near-equality}
 \int|\nabla h-a|^2\,\d\nu
 \leq 9D.
\end{equation}
\end{lemma}

\begin{proof}
Since $\int h\,\d\nu=0$, the Brascamp--Lieb inequality implies $D\geq0$.
The assertion is trivial when $h=0$, so we may assume
$s:=\norm{h}_{L^2(\nu)}>0$. 

Define $\tilde\nu=(x\mapsto\sqrt{\kappa}x)_\#\nu$ and
$u(z)=s^{-1}h(z/\sqrt\kappa)$. Then $\tilde\nu$ is $1$-strongly
log-concave and
\begin{equation}
	\int u\,\d\tilde\nu=0,\qquad \int u^2\,\d\tilde\nu=1,\qquad
	\int|\nabla u|^2\,\d\tilde\nu=1+\frac{D}{\kappa s^2}.
\end{equation}

Applying \cite[Lemma~2.2]{CourtadeFathi} with $k=1$, we obtain a unit vector
$w\in\R^n$ such that 
\begin{equation}
	 \int|\nabla u-w|^2\,\d\tilde\nu\leq \frac{9D}{\kappa s^2}. 
\end{equation}
Since $a=\int\nabla h\,\d\nu$ minimizes
$c\longmapsto\int|\nabla h-c|^2\,\d\nu$ over $c\in\R^n$,
\begin{equation}
	 \int|\nabla h-a|^2\,\d\nu
	\leq\int|\nabla h-\sqrt\kappa s\,w|^2\,\d\nu=
	\kappa s^2
	\int|\nabla u-w|^2
	\,\d\tilde\nu\leq9D.
\end{equation}
\end{proof}

\begin{proof}[Proof of Theorem~\ref{thm:equality}]
\medskip\noindent\textbf{Step 1: Alignment of the vector fields.}
Fix $t\in(0,1)$ and first suppose  
\begin{equation}
	C_P(\mu_t)
	=
	\left(
	\frac{1-t}{\sqrt{\kappa_0}}
	+\frac{t}{\sqrt{\kappa_1}}
	\right)^2. 
\end{equation}
Choose $\varphi_j\in C_c^\infty(\R^n)$ such that
\begin{equation}
	 \Var_{\mu_t}(\varphi_j)=1,\qquad
	\int|\nabla\varphi_j|^2\,\d\mu_t
	\rightarrow
	\left(\frac{1-t}{\sqrt{\kappa_0}}
	+\frac{t}{\sqrt{\kappa_1}}\right)^{-2}. 
\end{equation}
For $\varepsilon>0$, let $\mu_i^\varepsilon$ be the regularization of
$\mu_i$ given by Lemma~\ref{lem:regularization}, and let
$\mu_t^\varepsilon$ be the corresponding interpolated measure.
By the stability argument in the proof of Theorem~\ref{thm:main},
$\mu_t^\varepsilon\to\mu_t$ in $W_2$
as $\varepsilon\to0$. Since $\varphi_j$, $\varphi_j^2$, and
$|\nabla\varphi_j|^2$ are bounded and continuous, we may choose
$0<\varepsilon_j<1/j$ such that 
\begin{equation}
	\left|
		\Var_{\mu_t^{\varepsilon_j}}(\varphi_j)
	-
		\Var_{\mu_t}(\varphi_j)
	\right|
	\leq\frac{1}{j},\qquad \left|
	\int|\nabla\varphi_j|^2\,\d\mu_t^{\varepsilon_j}
	-
	\int|\nabla\varphi_j|^2\,\d\mu_t
	\right|\leq\frac{1}{j}. 
\end{equation}

Denote the regularized endpoints by 
$\mu_i^j:=\mu_i^{\varepsilon_j}$, $i=0,1$, and let $T_j$ be the
Brenier map from $\mu_0^j$ to $\mu_1^j$. Denote $ C_t=(1-t)\kappa_0^{-1/2}+t\kappa_1^{-1/2}$ and set 
\[
 H_j=DT_j,\qquad F_j=(1-t)I+tT_j,\qquad \mu_t^j:=\mu_t^{\varepsilon_j}=(F_j)_\#\mu_0^j. 
\]
For all sufficiently large $j$, define $f_j:=\bigl(\varphi_j-\int\varphi_j\,\d\mu_t^j\bigr)/
\sqrt{\Var_{\mu_t^j}(\varphi_j)}$. Then we have  
\begin{equation}\label{eq:intermediate-extremizers}
 \int f_j\,\d\mu_t^j=0,\qquad
 \int |f_j|^2\,\d\mu_t^j=1,\qquad
\int |\nabla f_j|^2\,\d\mu_t^j
 \rightarrow C_t^{-2}. 
\end{equation}
Set $g_j=f_j\circ F_j$ and $r_j=(\nabla f_j)\circ F_j$. Then we have  
\begin{equation}
	 \int g_j\,\d\mu_0^j=0, \qquad  \int |g_j|^2\,\d\mu_0^j=1,\qquad \int | r_j|^2\,\d\mu_0^j
	 \rightarrow C_t^{-2}. 
\end{equation}
Applying Proposition~\ref{prop:poisson} to $g_j$, we obtain a centered
function $u_j\in H^1(\mu_0^j)$.  Set
$X_j=H_j^{-1}\nabla u_j$ and $\mathsf Y_j=\nabla u_j$. Thus
$\mathsf Y_j$ is the target field from Proposition~\ref{prop:poisson}
pulled back by $T_j$, and
\begin{equation}\label{eq:regularized-fields}
 -\diver_{\mu_0^j}X_j=g_j,\qquad
 \norm{X_j}_{L^2(\mu_0^j)}\leq\kappa_0^{-1/2},\qquad
 \norm{\mathsf Y_j}_{L^2(\mu_0^j)}\leq\kappa_1^{-1/2}.
\end{equation}
All norms and inner products below are taken in $L^2(\mu_0^j)$.
Since $F_j$ is bi-Lipschitz, $g_j\in H^1(\mu_0^j)$ and is hence 
admissible as a test function in \eqref{eq:regularized-fields}. 
Since $\nabla g_j=((1-t)I+tH_j)r_j$,
\begin{equation}\label{eq:rigidity-sharpness-chain}
\begin{aligned}
		1=\int g_j^2\,\d\mu_0^j
	\overset{\eqref{eq:regularized-fields}}{=}\int X_j\cdot\nabla g_j\,\d\mu_0^j
	&=\int r_j\cdot\bigl((1-t)X_j+t\mathsf Y_j\bigr)\,\d\mu_0^j\\
	&\leq\norm{r_j}_{L^2(\mu_0^j)}
	\norm{(1-t)X_j+t\mathsf Y_j}_{L^2(\mu_0^j)}\\
	&\leq\norm{r_j}_{L^2(\mu_0^j)}
	\big((1-t)\norm{X_j}_{L^2(\mu_0^j)}+t\norm{\mathsf Y_j}_{L^2(\mu_0^j)}\big)\\
	&\leq C_t\norm{r_j}_{L^2(\mu_0^j)}\overset{\eqref{eq:intermediate-extremizers}}{\rightarrow}1, 
\end{aligned}
\end{equation}
The following chain is therefore asymptotically sharp, with
$\norm{r_j}_{L^2(\mu_0^j)}^{-1}\to C_t$:
\begin{equation}
	\norm{r_j}_{L^2(\mu_0^j)}^{-1}
	\leq\norm{(1-t)X_j+t\mathsf Y_j}_{L^2(\mu_0^j)}
	\leq(1-t)\norm{X_j}_{L^2(\mu_0^j)}+t\norm{\mathsf Y_j}_{L^2(\mu_0^j)}
	\leq C_t. 
\end{equation}
To make explicit which inequalities become sharp, set
\begin{align*}
 \delta_j^{\rm CS}
 &:=\norm{(1-t)X_j+t\mathsf Y_j}_{L^2(\mu_0^j)}
   -\norm{r_j}_{L^2(\mu_0^j)}^{-1},\\
 \delta_j^{\triangle}
 &:=(1-t)\norm{X_j}_{L^2(\mu_0^j)}
   +t\norm{\mathsf Y_j}_{L^2(\mu_0^j)}
   -\norm{(1-t)X_j+t\mathsf Y_j}_{L^2(\mu_0^j)},\\
 \delta_j^{\rm end}
 &:=C_t-(1-t)\norm{X_j}_{L^2(\mu_0^j)}
   -t\norm{\mathsf Y_j}_{L^2(\mu_0^j)}.
\end{align*}
These three quantities are nonnegative, and
\[
 \delta_j^{\rm CS}+\delta_j^{\triangle}+\delta_j^{\rm end}
 =C_t-\norm{r_j}_{L^2(\mu_0^j)}^{-1}\longrightarrow0.
\]
In particular, $\delta_j^{\rm end}\to0$. Since $0<t<1$ and the endpoint
norms are bounded above by $\kappa_0^{-1/2}$ and
$\kappa_1^{-1/2}$, respectively, we have
\[
 \delta_j^{\rm end}
 =(1-t)\bigl(\kappa_0^{-1/2}-\norm{X_j}_{L^2(\mu_0^j)}\bigr)
 +t\bigl(\kappa_1^{-1/2}-\norm{\mathsf Y_j}_{L^2(\mu_0^j)}\bigr),
\]
and each term on the right is nonnegative. Hence
\begin{equation}\label{eq:endpoint-field-norms}
		 \norm{X_j}_{L^2(\mu_0^j)}\rightarrow\kappa_0^{-1/2},\qquad
		\norm{\mathsf Y_j}_{L^2(\mu_0^j)}\to\kappa_1^{-1/2}.
\end{equation}
Expanding $\norm{(1-t)X_j+t\mathsf Y_j}_{L^2(\mu_0^j)}^2$ and combining with
\eqref{eq:endpoint-field-norms} yields 
\begin{equation}\label{eq:endpoint-field-inner-product}
 \ip{X_j}{\mathsf Y_j}_{L^2(\mu_0^j)}
 \rightarrow(\kappa_0\kappa_1)^{-1/2}.
\end{equation}
Therefore, we have  
\begin{equation}\label{eq:field-alignment}
	\begin{aligned}
		\norm{\mathsf Y_j-\sqrt{\frac{\kappa_0}{\kappa_1}}X_j}_{L^2(\mu_0^j)}^2
		=\norm{\mathsf Y_j}_{L^2(\mu_0^j)}^2
		+\frac{\kappa_0}{\kappa_1}\norm{X_j}_{L^2(\mu_0^j)}^2
		-2\sqrt{\frac{\kappa_0}{\kappa_1}}\ip{X_j}{\mathsf Y_j}_{L^2(\mu_0^j)}
		\rightarrow0.
	\end{aligned}
\end{equation}

\medskip\noindent\textbf{Step 2: Near-extremizers at the endpoints.}
We now use the limits of the vector fields to obtain near-extremizers at
the two endpoints.  Define $v_j=u_j\circ T_j^{-1}$.
Since $T_j$ is bi-Lipschitz and
$(T_j)_\#\mu_0^j=\mu_1^j$, we have
$v_j\in H^1(\mu_1^j)$ and
$\int v_j\,\d\mu_1^j=0$.  The Sobolev chain rule and change of variables yield
\begin{equation}\label{eq:target-gradient}
	(\nabla v_j)\circ T_j=X_j,\qquad
	\norm{v_j}_{L^2(\mu_1^j)}
	=\norm{u_j}_{L^2(\mu_0^j)}.
\end{equation}
Testing the weak divergence equation in
\eqref{eq:regularized-fields} with $u_j$, we obtain 
\begin{equation}
	 \int g_j u_j\,\d\mu_0^j
	=\int X_j\cdot\nabla u_j\,\d\mu_0^j
	=\ip{X_j}{\mathsf Y_j}_{L^2(\mu_0^j)}.
\end{equation}
Together with \eqref{eq:endpoint-field-inner-product}, this gives
\begin{equation}
		\left|\int g_j u_j\,\d\mu_0^j\right|
	=(\kappa_0\kappa_1)^{-1/2}+o(1).
\end{equation}
Since $\norm{g_j}_{L^2(\mu_0^j)}=1$, the endpoint Brascamp--Lieb inequalities and
\eqref{eq:target-gradient} yield
\begin{equation}
\begin{aligned}
	(\kappa_0\kappa_1)^{-1/2}+o(1)
	&\leq\norm{u_j}_{L^2(\mu_0^j)}\norm{g_j}_{L^2(\mu_0^j)}
	=\norm{v_j}_{L^2(\mu_1^j)}\\
	&\leq
	\min\left\{
	\kappa_0^{-1/2}\norm{\mathsf Y_j}_{L^2(\mu_0^j)},
	\kappa_1^{-1/2}\norm{X_j}_{L^2(\mu_0^j)}
	\right\}\\
	&=(\kappa_0\kappa_1)^{-1/2}+o(1), 
\end{aligned}
\end{equation}
which implies $\norm{u_j}_{L^2(\mu_0^j)}=\norm{v_j}_{L^2(\mu_1^j)}
\rightarrow(\kappa_0\kappa_1)^{-1/2}$ and the endpoint Poincar\'e deficits
\begin{equation}
	\delta_{0j}
	=\int|\nabla u_j|^2\,\d\mu_0^j
	-\kappa_0\int u_j^2\,\d\mu_0^j\rightarrow 0,
	\qquad
	\delta_{1j}
	=\int|\nabla v_j|^2\,\d\mu_1^j
	-\kappa_1\int v_j^2\,\d\mu_1^j\rightarrow 0.
\end{equation}

Let $p_j=\int\nabla u_j\,\d\mu_0^j$ and
$q_j=\int\nabla v_j\,\d\mu_1^j$. Since the regularized potentials have
the same curvature lower bounds as the original endpoints,
Lemma~\ref{lem:near-equality} gives
\begin{equation}\label{eq:near-affine-gradient}
\nabla u_j-p_j\rightarrow0
\quad\text{in }L^2(\mu_0^j),
\qquad
\nabla v_j-q_j\rightarrow0
\quad\text{in }L^2(\mu_1^j).
\end{equation}
Since $p_j=\int \mathsf Y_j\,\d\mu_0^j$ and
$q_j=\int X_j\,\d\mu_0^j$, the orthogonal decompositions
\begin{equation}
	 \norm{\mathsf Y_j}_{L^2(\mu_0^j)}^2
	=\norm{\mathsf Y_j-p_j}_{L^2(\mu_0^j)}^2+|p_j|^2,
	\qquad
	\norm{X_j}_{L^2(\mu_0^j)}^2
	=\norm{X_j-q_j}_{L^2(\mu_0^j)}^2+|q_j|^2 
\end{equation}
imply
$|p_j|\to\kappa_1^{-1/2}$ and $|q_j|\to\kappa_0^{-1/2}$. Combining this
with \eqref{eq:field-alignment}, we obtain
\begin{equation}
	\left|p_j-\sqrt{\frac{\kappa_0}{\kappa_1}}q_j\right|
	\leq \norm{\mathsf Y_j-\sqrt{\frac{\kappa_0}{\kappa_1}}X_j}_{L^2(\mu_0^j)}
	\rightarrow 0. 
\end{equation}
Since $|\sqrt{\kappa_1}\,p_j|\to1$ and
$\sqrt{\kappa_1}\,p_j-\sqrt{\kappa_0}\,q_j\to0$, after passing to a
subsequence, there exists $e\in\R^n$ with $|e|=1$ such that
\begin{equation}\label{eq:common-direction}
	\sqrt{\kappa_1}\,p_j\to e,
	\qquad
	\sqrt{\kappa_0}\,q_j\to e.
\end{equation}

\medskip\noindent\textbf{Step 3: Affine limits and Gaussian integration by parts.}
Let $m_i^j=\int x\,\d\mu_i^j$. By
\eqref{eq:near-affine-gradient} and the endpoint Poincar\'e inequalities,
\begin{equation}\label{eq:affine-approximation}
 \norm{u_j-p_j\cdot(x-m_0^j)}_{H^1(\mu_0^j)}\rightarrow0,\qquad
 \norm{v_j-q_j\cdot(x-m_1^j)}_{H^1(\mu_1^j)}\rightarrow0.
\end{equation}
For $i=0,1$ and $r,s\in H^1(\mu_i^j)$, consider the symmetric
positive-semidefinite bilinear form
\[
 Q_i^{(j)}(r,s)
 =\int\nabla r\cdot\nabla s\,\d\mu_i^j
  -\kappa_i\Cov_{\mu_i^j}(r,s).
\]
The Poincar\'e inequality makes $Q_i^{(j)}$ positive semidefinite, so it
satisfies the Cauchy--Schwarz inequality. For fixed
$\psi\in C_c^\infty(\R^n)$,
\begin{equation}
	0\leq Q_i^{(j)}(\psi,\psi)
	\leq\int|\nabla\psi|^2\,\d\mu_i^j
	\leq\norm{\nabla\psi}_{L^\infty}^2,\qquad \text{uniformly in $j$}, 
\end{equation}
which implies  
\begin{equation}
	|Q_0^{(j)}(u_j,\psi)|^2
	\leq\delta_{0j}Q_0^{(j)}(\psi,\psi)\rightarrow0,\qquad 	|Q_1^{(j)}(v_j,\psi)|^2
	\leq\delta_{1j}Q_1^{(j)}(\psi,\psi)\rightarrow0. 
\end{equation}
Let $\ell_{0j}=p_j\cdot(x-m_0^j)$. By
\eqref{eq:affine-approximation},
\begin{equation}
		|Q_0^{(j)}(u_j-\ell_{0j},\psi)|^2\leq\norm{\nabla(u_j-\ell_{0j})}_{L^2(\mu_0^j)}^2
		Q_0^{(j)}(\psi,\psi)\to0.
\end{equation}
Hence we have  
\begin{equation}
	 Q_0^{(j)}(\ell_{0j},\psi)
	=Q_0^{(j)}(u_j,\psi)
	-Q_0^{(j)}(u_j-\ell_{0j},\psi)
	\rightarrow0, 
\end{equation}
and the analogous conclusion holds at the target.  Expanding the
source identity yields 
\begin{equation}
 Q_0^{(j)}\bigl(p_j\cdot(x-m_0^j),\psi\bigr)
=p_j\cdot\left[
\int\nabla\psi\,\d\mu_0^j
-\kappa_0\int(x-m_0^j)\psi\,\d\mu_0^j
\right],
\end{equation}
and the target identity is the same with $q_j,\mu_1^j,\kappa_1$.
Moreover, for fixed $\psi\in C_c^\infty(\R^n)$, the functions
$\nabla\psi$, $\psi$, and $x\psi(x)$ are bounded and continuous.
Thus every term in the brackets converges under weak convergence, while
$m_i^j\to m_i$ follows from $W_2$ convergence. Combining with 
\eqref{eq:common-direction} and passing to the limit, we obtain 
\begin{equation}\label{eq:Gaussian-IBP}
 \int e\cdot\nabla\psi\,\d\mu_i
 =\kappa_i\int e\cdot(x-m_i)\psi(x)\,\d\mu_i,
 \qquad \psi\in C_c^\infty(\R^n),\quad i=0,1.
\end{equation}

\medskip\noindent\textbf{Step 4: Gaussian factorization at both endpoints.}
We finish the necessity by identifying the measures satisfying
\eqref{eq:Gaussian-IBP}.  For $i=0,1$, define the deterministic coordinate maps
\[
 s_i(x)=e\cdot(x-m_i),
 \qquad
 z_i(x)=P_{e^\perp}(x-m_i).
\]
If $X_i\sim\mu_i$, then $S_i=s_i(X_i)$ and $Z_i=z_i(X_i)$ are the
corresponding random coordinates.  Their joint characteristic function is
\begin{equation}
	 \phi_i(\tau,\xi)
	=\mathbb E\,e^{\,\mathrm{i}(\tau S_i+\xi\cdot Z_i)}
	=\int e^{\,\mathrm{i}(\tau s_i(x)+\xi\cdot z_i(x))}\,\d\mu_i(x),
	\qquad \tau\in\R,\quad \xi\in e^\perp. 
\end{equation}

Choose a radial cut-off function $\chi\in C_c^\infty(\R^n)$ such that
$0\leq\chi\leq1$, $\chi=1$ on $B_1$, and
$\operatorname{supp}\chi\subseteq B_2$. We may apply 
\eqref{eq:Gaussian-IBP} to the real and imaginary parts of $\psi_R:=\chi\!\left((x-m_i)/R\right)
e^{\,\mathrm{i}(\tau s_i(x)+\xi\cdot z_i(x))}$ and then combine the resulting identities. Note that 
\begin{equation}
	e\cdot\nabla\psi_R(x)
=
\frac1R
e\cdot\nabla\chi\!\left(\frac{x-m_i}{R}\right)
e^{\,\mathrm{i}(\tau s_i(x)+\xi\cdot z_i(x))}+\mathrm{i}\tau\chi\!\left(\frac{x-m_i}{R}\right)e^{\,\mathrm{i}(\tau s_i(x)+\xi\cdot z_i(x))}.
\end{equation}
Applying \eqref{eq:Gaussian-IBP} to $\psi_R$ and then passing to the limit, we obtain 
\begin{equation}\label{eq:characteristic-function-IBP}
	\mathrm{i}\tau\phi_i(\tau,\xi)
	=
	\kappa_i
	\int
	s_i(x)e^{\,\mathrm{i}(\tau s_i(x)+\xi\cdot z_i(x))}
	\,\d\mu_i(x)=\kappa_i \mathrm{i}^{-1}\partial_\tau\phi_i(\tau,\xi), 
\end{equation}
which implies
$\phi_i(\tau,\xi)=e^{-\tau^2/(2\kappa_i)}\phi_i(0,\xi)$. By uniqueness of characteristic functions, $S_i\sim N(0,\kappa_i^{-1})$ and is independent of $Z_i$.  Applying
this argument to $i=0,1$ proves \eqref{eq:factor} with the same
direction $e$.

\medskip\noindent\textbf{Step 5: The converse and equality at all times.}
Conversely, suppose that \eqref{eq:factor} holds. Let $G_i$ denote the Gaussian
factor in \eqref{eq:factor}. Note that the quadratic cost is additive under
$\R e\oplus e^\perp$. Hence the product of optimal couplings between
$(G_0,G_1)$ and $(\nu_0,\nu_1)$ is optimal between the full endpoints.
Since $\mu_0$ is absolutely continuous, the quadratic optimal transport
plan is unique and therefore coincides with this product plan. Thus the
optimal transport splits along $\R e\oplus e^\perp$. 

Let $(X_0,X_1)$ have the optimal joint law and set
$S_i=e\cdot(X_i-m_i)$ for $i=0,1$. Then $S_0,S_1$ are the centered Gaussian
coordinates of the two marginals.  Moreover, the optimal Gaussian coupling satisfies
$S_1=\sqrt{\kappa_0/\kappa_1}\,S_0$ almost surely. At time $t\in(0,1)$, let
\[
\ell_t(x)=e\cdot(x-m_t),
\qquad
m_t=\int x\,\d\mu_t.
\]
Set $X_t=(1-t)X_0+tX_1$, so that $X_t\sim\mu_t$ and
$m_t=(1-t)m_0+tm_1$. Hence
\begin{equation}
	\ell_t(X_t)
	=(1-t)S_0+tS_1
	=
	\left(
	(1-t)+t\sqrt{\frac{\kappa_0}{\kappa_1}}
	\right)S_0.
\end{equation}
Consequently, since $\Var(S_0)=\kappa_0^{-1}$ and
$\nabla\ell_t=e$ with $|e|=1$, we obtain
\begin{equation}
	\Var_{\mu_t}(\ell_t)
	=
	\left(
	\frac{1-t}{\sqrt{\kappa_0}}
	+\frac{t}{\sqrt{\kappa_1}}
	\right)^2,
	\qquad
	\int|\nabla\ell_t|^2\,\d\mu_t=1.
\end{equation}

Although $\ell_t$ is not compactly supported, its standard cutoffs $\ell_{t,R}(x)=\chi(x/R)\ell_t(x)$,
where $\chi\in C_c^\infty(\R^n)$ is radial, $\chi=1$ on $B_1$, and $\operatorname{supp}\chi\subseteq B_2$, satisfy
\begin{equation}\label{eq:linear-cutoff-convergence}
	\ell_{t,R}\longrightarrow\ell_t
	\qquad\text{in } H^1(\mu_t). 
\end{equation}
Indeed,
\begin{equation}
	\nabla\ell_{t,R}
	=
	\chi(x/R)e
	+
	R^{-1}\ell_t(x)(\nabla\chi)(x/R), 
\end{equation}
and the convergence follows from $\mu_t\in\mathcal P_2(\R^n)$. 

Therefore, by the definition of the Poincar\'e constant and the
convergences in \eqref{eq:linear-cutoff-convergence},
\begin{equation}\label{eq:Gaussian-factor-Rayleigh}
	C_P(\mu_t)
	\geq
	\lim_{R\to\infty}
	\frac{\Var_{\mu_t}(\ell_{t,R})}
	{\int|\nabla\ell_{t,R}|^2\,\d\mu_t}
	=
	\left(\frac{1-t}{\sqrt{\kappa_0}}
	+\frac{t}{\sqrt{\kappa_1}}\right)^2.
\end{equation}
Combining with Theorem~\ref{thm:main} proves the result. Replacing $t$ by
any $s\in[0,1]$ in \eqref{eq:linear-cutoff-convergence} and
\eqref{eq:Gaussian-factor-Rayleigh} proves equality at every time.
\end{proof}

The equality argument identifies all common Gaussian directions, not
only a single one.  This yields the following maximal factorization.

\begin{proposition}[Maximal common Gaussian factor]
\label{prop:maximal-Gaussian-factor}
Let $\mu_i$ be $\kappa_i$-strongly log-concave probability measures
on $\R^n$, let $m_i=\int x\,\d\mu_i$, and define
\begin{equation}\label{eq:maximal-Gaussian-subspace}
	E=\left\{v\in\R^n:
		\int v\cdot\nabla\psi\,\d\mu_i
		=\kappa_i\int v\cdot(x-m_i)\psi\,\d\mu_i,\quad \forall\,\psi\in C_c^\infty(\R^n),\quad i=0,1\right\}.
\end{equation}
Then $E$ is a linear subspace. Moreover, under the orthogonal identification
$\R^n=E\oplus E^\perp$, there exist $\nu_i\in\mathcal P(E^\perp)$ such
that
\begin{equation}\label{eq:maximal-factor}
 \mu_i=N_E(P_E m_i,\kappa_i^{-1}I_E)\otimes\nu_i,
 \qquad i=0,1.
\end{equation}
Here $N_E$ denotes a Gaussian law on $E$.  If a subspace $F\subseteq
\R^n$ admits such a factorization for both endpoints, with respective
covariances $\kappa_i^{-1}I_F$, then $F\subseteq E$.  Thus $E$ is the
unique maximal common Gaussian subspace.  Moreover, equality in
Theorem~\ref{thm:main} at one (equivalently, every) time
$t\in(0,1)$ holds if and only if $E\neq\{0\}$.
\end{proposition}

\begin{proof}
The defining identities in \eqref{eq:maximal-Gaussian-subspace} are
linear in $v$, so $E$ is a linear subspace. Fix $i\in\{0,1\}$ and set
$s=P_E(x-m_i)$ and $z=P_{E^\perp}(x-m_i)$. For
$\tau\in E$ and $\xi\in E^\perp$, let
$\phi_i(\tau,\xi)=\int e^{\,\mathrm{i}(\tau\cdot s+\xi\cdot z)}\,\d\mu_i$.
Applying \eqref{eq:maximal-Gaussian-subspace} to cut-off
approximations of the exponential, exactly as in the proof of
Theorem~\ref{thm:equality}, we obtain 
\begin{equation}
	v\cdot\nabla_\tau\phi_i(\tau,\xi)
	=-\kappa_i^{-1}(v\cdot\tau)\phi_i(\tau,\xi),\qquad \text{for every $v\in E$},
\end{equation}
which implies
$\nabla_\tau\phi_i(\tau,\xi)=-\kappa_i^{-1}\tau\phi_i(\tau,\xi)$, and hence  $ \phi_i(\tau,\xi)
=e^{-|\tau|^2/(2\kappa_i)}\phi_i(0,\xi)$. By uniqueness of characteristic functions, this factorization shows that
$s\sim N_E(0,\kappa_i^{-1}I_E)$ and is independent of $z$; after
recentering, this yields \eqref{eq:maximal-factor}. 

Conversely, if both endpoints split off Gaussian factors on a subspace
$F$ with covariances $\kappa_i^{-1}I_F$, then Gaussian integration by
parts shows that every $v\in F$ satisfies
\eqref{eq:maximal-Gaussian-subspace}. Hence $F\subseteq E$, so $E$ is
the unique maximal common Gaussian subspace. Finally,
Theorem~\ref{thm:equality} shows that equality at an interior time is
equivalent to $E\neq\{0\}$. 
\end{proof}

Proposition~\ref{prop:maximal-Gaussian-factor} gives the following
geometric description in the extended-valued case. For a closed convex
set $K$, let $\operatorname{Lin}(K):=\{v:K+\R v=K\}$ denote its lineality space. 
\begin{corollary}[Conditioned Gaussian measures]
\label{cor:conditioned-gaussians}
Let $K_i\subseteq\R^n$ be convex Borel sets with nonempty interior, let
$\kappa_i>0$, and let $\mu_i$ be the probability measures with densities
\[
 \d\mu_i(x)
 \propto e^{-\kappa_i|x|^2/2}\mathbf1_{K_i}(x)\,\d x,
 \qquad i=0,1.
\]
Then their quadratic Wasserstein interpolation satisfies
\begin{equation}
	 C_P(\mu_t)
	\leq\left(\frac{1-t}{\sqrt{\kappa_0}}
	+\frac{t}{\sqrt{\kappa_1}}\right)^2,
	\qquad 0\leq t\leq1. 
\end{equation}
Equality holds for some $t\in(0,1)$ if and only if
\begin{equation}
	\operatorname{Lin}(\bar K_0)
	\cap\operatorname{Lin}(\bar K_1)\neq\{0\}.
\end{equation}
In that case equality holds for every $t\in[0,1]$.
In particular, if either $K_0$ or $K_1$ is bounded, then
\begin{equation}
	 C_P(\mu_t)
	<\left(\frac{1-t}{\sqrt{\kappa_0}}
	+\frac{t}{\sqrt{\kappa_1}}\right)^2,
	\qquad\text{for every }t\in(0,1). 
\end{equation}
\end{corollary}

\begin{proof}
Since the boundary of a convex set with nonempty interior has
Lebesgue measure zero, replacing $K_i$ by $\bar K_i$ does not change $\mu_i$.
 We may assume that both sets are closed. Up to additive constants, the extended-valued endpoint potentials are
$V_i(x)=\frac{\kappa_i}{2}|x|^2+\iota_{K_i}(x)$,
where $\iota_{K_i}$ is zero on $K_i$ and $+\infty$ on $K_i^c$.  Hence
Theorem~\ref{thm:main} implies the asserted inequality.

Suppose first that the two lineality spaces contain a common unit
vector $e$.  Denote $K_i'=K_i\cap e^\perp$.  We claim that $K_i=\R e+K_i'$.
Indeed, $K_i+\R e=K_i$, and subtracting the $e$-component of any
point of $K_i$ leaves a point of $K_i'$.  The Gaussian density
therefore factorizes under $\R^n=\R e\oplus e^\perp$, and hence
$\mu_i=N(0,\kappa_i^{-1})\otimes\nu_i$, where $\nu_i$ is the
corresponding conditioned Gaussian measure on $K_i'$. The converse
part of Theorem~\ref{thm:equality} yields equality at
every time.

Conversely, by Theorem~\ref{thm:equality}, equality at an interior time gives a common Gaussian line
factor.  Since the density of $\mu_i$ is positive
on $\operatorname{int}K_i$ and
$K_i=\operatorname{cl}(\operatorname{int}K_i)$, its support is $K_i$. On the
other hand, the support of the product in \eqref{eq:factor} has the
form $\R e+\operatorname{supp}\nu_i$. It is invariant under translation
by $\R e$, and therefore
$e\in\operatorname{Lin}(K_i)$ for both endpoints.  Finally, a bounded
convex set has trivial lineality space, which proves the strict
inequality in the last assertion.
\end{proof}

\section{Gaussian entropy and functional inequalities}\label{sec:Gaussian-applications}
We now apply the Poincar\'e interpolation theorem to the second variation
of Gaussian relative entropy.
Throughout this section, $\gamma$ denotes the standard Gaussian measure,
and $\mathcal P_2(\R^n)$ denotes the probability measures on $\R^n$
with finite second moment.  We write
\[
L=\Delta-x\cdot\nabla,
\qquad
\Gamma_2(u)=\norm{\Hess{u}}_{\HS}^2+|\nabla u|^2
\]
for the Ornstein--Uhlenbeck generator and its iterated carr\'e du champ,
respectively. We also write $P_s=e^{sL}$ for the associated Markov
semigroup on functions and $P_s^*$ for its dual action on probability
measures.

All the entropies appearing below are finite.  Indeed, relative to
$\gamma$, a $1$-strongly log-concave endpoint has density
$e^{-W}$ with $W$ convex, possibly extended-valued. Since $W$ is
bounded below by an affine function, its negative part
grows at most linearly, while
$s e^{-s}\leq e^{-1}$ for $s\geq0$ controls its positive part.
Thus $\Ent_\gamma(\mu_i)<\infty$.  Moreover,
\[
\Ent_\gamma(\mu)
=
\Ent_{\mathcal L^n}(\mu)
+\frac12\int|x|^2\,\d\mu
+\frac n2\log(2\pi).
\]
The second moment remains finite along the Wasserstein geodesic, and
$\Ent_{\mathcal L^n}$ is displacement convex
\cite{McCannConvexity}.  Hence
$\Ent_\gamma(\mu_t)<\infty$ for every $t\in[0,1]$.

When the mean is zero, the estimate below follows the argument in
\cite[Theorem~4.6]{AishwaryaRotem}, which subtracts a quadratic
polynomial.  In
Lemma~\ref{lem:mean-Gamma2}, we separate the translation component.  This
removes the zero-average and symmetry assumptions. The local-to-global
comparison and the subsequent functional estimates follow the arguments used by Aishwarya and Rotem
\cite[Sections~4--5]{AishwaryaRotem} and by Aishwarya and Li
\cite[proof of Theorem~1.7]{AishwaryaLi} to derive functional
inequalities from entropy. We prove the changes required by the translation
component and cite each unchanged step at the point where it is used.

\begin{lemma}[$\Gamma_2$ estimate with the mean separated]
\label{lem:mean-Gamma2}
Let $\mu\in\mathcal P_2(\R^n)$ satisfy $C_P(\mu)\leq1$, and let
$u\in C^\infty(\R^n)$ satisfy
$\nabla u,\Hess{u}\in L^2(\mu)$. Denote
\[
 m=\int x\,\d\mu,
 \qquad
 a=\int\nabla u\,\d\mu,
 \qquad
 \ell=\int Lu\,\d\mu. 
\]
Then
\begin{equation}\label{eq:mean-Gamma2}
 \int\Gamma_2(u)\,\d\mu
 \geq
 2\int|\nabla u|^2\,\d\mu-|a|^2
 +\frac{(\ell+m\cdot a)^2}{n}+\frac{(\ell+m\cdot a)^2}{n^2}\int|x-m|^2\,\d\mu. 
\end{equation}
\end{lemma}

\begin{proof}
The proof follows \cite[Theorem~4.6]{AishwaryaRotem}; we give the
modifications involving $m$ and $a$.
First suppose that $a=0$. Set $\lambda=\ell/n$ and define
$v(x)=u(x)-\frac{\lambda}{2}|x-m|^2$.
Then $\int\nabla v\,\d\mu=0$. The derivative $\partial_j v$ need not be
compactly supported.  We therefore apply the Poincar\'e inequality to
$\eta_R\partial_j v$, where $\eta_R$ is a standard radial cut-off, and
then let $R\to\infty$.  This is the same cut-off argument used in the
proof of Theorem~\ref{thm:equality}. The error
terms vanish since $\nabla v,\Hess{v}\in L^2(\mu)$. Summing over
$j=1,\ldots,n$, gives
\begin{equation}\label{eq:component-Poincare}
	\int\norm{\Hess{v}}_{\HS}^2\,\d\mu
	\geq\int|\nabla v|^2\,\d\mu.
\end{equation}
 Expanding \eqref{eq:component-Poincare} and using
$\nabla v=\nabla u-\lambda(x-m)$ and $\Hess{v}=\Hess{u}-\lambda I$, we obtain
\begin{equation}\label{eq:mean-Hessian-expansion}
	\begin{aligned}
	\int\norm{\Hess{u}}_{\HS}^2\,\d\mu
	-2\lambda\int\Delta u\,\d\mu+n\lambda^2
	&\geq
	\int|\nabla u|^2\,\d\mu
	\\
	&\quad-2\lambda\int(x-m)\cdot\nabla u\,\d\mu
	+\lambda^2\int|x-m|^2\,\d\mu.
	\end{aligned}
\end{equation}
Since $a=0$, $ \int\bigl(\Delta u-(x-m)\cdot\nabla u\bigr)\,\d\mu
=\ell$. Thus,
\begin{equation}
\begin{aligned}
 \int\norm{\Hess{u}}_{\HS}^2\,\d\mu
 &\geq
 \int|\nabla u|^2\,\d\mu
 +2\lambda\ell-n\lambda^2+\lambda^2\int|x-m|^2\,\d\mu\\
 &=
 \int|\nabla u|^2\,\d\mu
 +\frac{\ell^2}{n}+\frac{\ell^2}{n^2}\int|x-m|^2\,\d\mu, 
\end{aligned}
\end{equation}
which implies \eqref{eq:mean-Gamma2} when $a=0$.

For general $a$, apply the case $a=0$ to $w(x)=u(x)-a\cdot x$. Then
\begin{equation}
	\int\nabla w\,\d\mu=0,
	\qquad
	\int Lw\,\d\mu=\ell+m\cdot a,
	\qquad
	\int|\nabla w|^2\,\d\mu
	=\int|\nabla u|^2\,\d\mu-|a|^2.
\end{equation}
Since $\Hess{w}=\Hess{u}$,
\begin{equation}
	\int\Gamma_2(w)\,\d\mu
	=
	\int\Gamma_2(u)\,\d\mu-|a|^2. 
\end{equation}
Applying the estimate for the case $a=0$ to $w$ then yields
\eqref{eq:mean-Gamma2}.
\end{proof}

\begin{proof}[Proof of Theorem~\ref{thm:centered-entropy}]
\medskip\noindent\textbf{Step 1: Entropy variations.}
First assume that the endpoint potentials are smooth with two-sided
Hessian bounds.  Let $\nabla\theta_t$ be the velocity of the
constant-speed geodesic, so
\begin{equation}
	\partial_t\mu_t+\diver(\mu_t\nabla\theta_t)=0,
	\qquad
	\partial_t\theta_t+\frac12|\nabla\theta_t|^2=0.
\end{equation}

Write $\mu_t=\rho_t\gamma$ and set
$E_t:=\Ent_\gamma(\mu_t)=\int\rho_t\log\rho_t\,\d\gamma$.
The first- and second-variation formulas in
\cite[Lemma~4.3]{AishwaryaRotem} yield
\begin{equation}\label{eq:entropy-first-variation}
	E_t'=\int\nabla\log\rho_t\cdot\nabla\theta_t\,\d\mu_t=\int\nabla\rho_t\cdot\nabla\theta_t\,\d\gamma
=-\int L\theta_t\,\d\mu_t,\qquad  E_t''=\int\Gamma_2(\theta_t)\,\d\mu_t, 
\end{equation}
where the assumed regularity and Gaussian tails justify the differentiations
and integrations by parts above via standard radial cutoffs.

\medskip\noindent\textbf{Step 2: Separating the translation component.}
Set $m_t=\int x\,\d\mu_t$ and $a=m_1-m_0$. The continuity equation and
the constant-speed property of the geodesic imply
\begin{equation}\label{eq:geodesic-moments}
 m_t=(1-t)m_0+tm_1,
 \qquad
 \int\nabla\theta_t\,\d\mu_t=m_t'=a,
 \qquad
 \int|\nabla\theta_t|^2\,\d\mu_t=W_2^2(\mu_0,\mu_1).
\end{equation}
Here $a$ is precisely the translation velocity. Set
$W_{\rm c}^2=W_2^2(\mu_0,\mu_1)-|a|^2$. Since both endpoints are
$1$-strongly log-concave, Theorem~\ref{thm:main} gives
$C_P(\mu_t)\leq1$. By \eqref{eq:entropy-first-variation} and
\eqref{eq:geodesic-moments}, the function
$\mathcal D_t=E_t-|m_t|^2/2$ satisfies
\begin{equation}
	 \mathcal D_t'=E_t'-m_t\cdot a,\qquad  \mathcal D_t''=E_t''-|a|^2,
\end{equation}
and 
\begin{equation}
	 \int L\theta_t\,\d\mu_t+m_t\cdot a
	=-\mathcal D_t',\qquad \int|\nabla\theta_t|^2\,\d\mu_t-|a|^2
	=W_{\rm c}^2. 
\end{equation}
Applying Lemma~\ref{lem:mean-Gamma2} to $\theta_t$, we obtain
\begin{equation}\label{eq:centered-entropy-refined}
	\mathcal D_t''
	\geq
	2W_{\rm c}^2+\frac{(\mathcal D_t')^2}{n}
	+\frac{(\mathcal D_t')^2}{n^2}\int|x-m_t|^2\,\d\mu_t, 
\end{equation}
which implies \eqref{eq:intro-entropy-differential}. 

\medskip\noindent\textbf{Step 3: Global comparison.}
For $i=0,1$, let $\bar\mu_i=(x\mapsto x-m_i)_\#\mu_i$
be the centered translate of $\mu_i$. For every coupling $\pi$ of the endpoints, 
\begin{equation}\label{eq:centered-cost-decomposition}
	 \int|x-y|^2\,\d\pi
	=|m_1-m_0|^2
	+\int|(x-m_0)-(y-m_1)|^2\,\d\pi, 
\end{equation}
which implies $W_{\rm c}=W_2(\bar\mu_0,\bar\mu_1)$. Translating the
optimal transport plan also yields the centered interpolation
$\bar\mu_t=(x\mapsto x-m_t)_\#\mu_t$.  The Brascamp--Lieb inequality yields $\Cov(\bar\mu_i)\preceq I$.  If $\bar X_0$ and $\bar X_1$ are
independent with these laws, then
\begin{equation}
	W_{\rm c}^2
	\leq\mathbb E|\bar X_0-\bar X_1|^2
	=\tr\Cov(\bar\mu_0)+\tr\Cov(\bar\mu_1)
	\leq2n. 
\end{equation}
Thus $\omega=\sqrt{2/n}\,W_{\rm c}\leq2<\pi$. Setting
$U_t=e^{-\mathcal D_t/n}$, \eqref{eq:centered-entropy-refined} gives
\begin{equation}\label{eq:entropy-ODE}
	U_t''+\omega^2U_t\leq0. 
\end{equation}
The one-dimensional comparison principle then yields
\eqref{eq:intro-entropy-distortion}. This is the standard
local-to-global comparison for $(2,n)$-convexity; see
\cite[Lemma~5.3]{AishwaryaRotem} and
\cite[Lemma~2.2]{ErbarKuwadaSturm}. 

\medskip\noindent\textbf{Step 4: Equality in the smooth case.}
Keeping the last term in \eqref{eq:centered-entropy-refined} yields
\begin{equation}\label{eq:entropy-gap-forcing}
	-U_t''-\omega^2U_t
	\geq
	\frac{U_t(\mathcal D_t')^2}{n^3}\int|x-m_t|^2\,\d\mu_t. 
\end{equation}

For $0\leq\omega<\pi$, the Dirichlet Green kernel $G_\omega$ of
$-\frac{\d^2}{\d t^2}-\omega^2$ on $(0,1)$ is given by 
\begin{equation}
	G_\omega(t,s)
	=
	\frac{\sin(\omega t)\sin(\omega(1-s))}
	{\omega\sin\omega},\qquad 0<t\leq s<1, 
\end{equation}
and $G_\omega(t,s)=G_\omega(s,t)$ when $s<t$. At $\omega=0$, the
right-hand side is understood by continuity and equals $t(1-s)$.
In particular, $G_\omega(t,s)>0$ for all $t,s\in(0,1)$. 

Since the sine interpolation below solves the homogeneous equation
$v''+\omega^2v=0$ and agrees with $U$ at $t=0,1$, the Green
representation yields, for every $t_0\in(0,1)$, 
\begin{equation}\label{eq:entropy-Green-representation}
		U_{t_0}-\frac{\sin((1-t_0)\omega)}{\sin\omega}U_0
	-\frac{\sin(t_0\omega)}{\sin\omega}U_1=
	\int_0^1G_\omega(t_0,s)
	\bigl(-U_s''-\omega^2U_s\bigr)\,\d s
	\geq0. 
\end{equation}

Suppose that equality holds at some $t_0\in(0,1)$. Since
$G_\omega(t_0,s)>0$,
\begin{equation}
	-U_s''-\omega^2U_s=0,
	\qquad \text{for almost every }s\in(0,1). 
\end{equation}
Since $\mu_s$ is not a point mass, $\int|x-m_s|^2\,\d\mu_s>0$. 
Combining with \eqref{eq:entropy-gap-forcing} yields 
$\mathcal D_s'=0$ for almost every $s\in(0,1)$. Under the present
smoothness assumptions, $\mathcal D'$ is continuous, and hence
$\mathcal D_s'=0$ for every $s\in(0,1)$.  It follows from
\eqref{eq:centered-entropy-refined} that $W_{\rm c}=0$.  Hence $\mu_1$
is a translate of $\mu_0$.
Conversely, if $\mu_1$ is a translate of $\mu_0$, then the centered
geodesic is constant. Hence $W_{\rm c}=0$ and $\mathcal D_t$ is
constant, so equality holds for every $t\in[0,1]$. 

\medskip\noindent\textbf{Step 5: Regularization. }
It remains to remove the smoothness assumptions. Denote 
$s_\varepsilon=\frac12\log(1+\varepsilon)$ and define
$\bar\mu_i^\varepsilon=P_{s_\varepsilon}^*\bar\mu_i$.
By Lemma~\ref{lem:regularization}, the measures
$\bar\mu_i^\varepsilon$ are centered, smooth, $1$-strongly
log-concave and $\bar\mu_i^\varepsilon\rightarrow\bar\mu_i$ in $W_2$. Let $(\bar\mu_t^\varepsilon)_{t\in[0,1]}$ denote their quadratic Wasserstein geodesic, set $W_\varepsilon
:=W_2(\bar\mu_0^\varepsilon,\bar\mu_1^\varepsilon)$ and $\omega_\varepsilon
:=\sqrt{2/n}\,W_\varepsilon$. 

Since $P_{s_\varepsilon}^*\gamma=\gamma$, the data processing
inequality yields $\Ent_\gamma(\bar\mu_i^\varepsilon)
\leq \Ent_\gamma(\bar\mu_i)$. 
On the other hand, lower semicontinuity of relative entropy gives
$\Ent_\gamma(\bar\mu_i)
\leq
\liminf_{\varepsilon\to0}
\Ent_\gamma(\bar\mu_i^\varepsilon)$; this follows directly from the
Donsker--Varadhan variational formula \cite[Section~2]{DonskerVaradhan}.
Hence $\Ent_\gamma(\bar\mu_i^\varepsilon)
\longrightarrow
\Ent_\gamma(\bar\mu_i)$. Stability of optimal transport plans, as in
the proof of Theorem~\ref{thm:main}, gives
$\bar\mu_t^\varepsilon\to\bar\mu_t$ in $W_2$ and
$W_\varepsilon\to W_{\rm c}$.  Apply the smooth inequality to
$(\bar\mu_t^\varepsilon)_{t\in[0,1]}$.  The endpoint entropies converge,
and entropy is lower semicontinuous at time $t$.  We may therefore pass
to the limit and obtain the same inequality for
$\Ent_\gamma(\bar\mu_t)$. Since
\begin{equation}
	\Ent_\gamma(\bar\mu_t)
	=
	\Ent_\gamma(\mu_t)-\frac12|m_t|^2
	=\mathcal D_t,
\end{equation}
this is precisely \eqref{eq:intro-entropy-distortion}. 

For the equality case, set
\[
 U_s^\varepsilon
 :=e^{-\Ent_\gamma(\bar\mu_s^\varepsilon)/n}
\]
and let $S_\varepsilon(t_0)$ denote the sine interpolation of
$U_0^\varepsilon$ and $U_1^\varepsilon$ with frequency
$\omega_\varepsilon$. Thus the gap for the regularized geodesic is
$\Delta_\varepsilon(t_0)=U_{t_0}^\varepsilon-S_\varepsilon(t_0)\geq0$.
The convergence of the endpoint entropies and of $\omega_\varepsilon$
gives $S_\varepsilon(t_0)\to S(t_0)$, the corresponding sine
interpolation for the limiting endpoints. Lower semicontinuity at the
interior time gives
\begin{equation}
	 \Ent_\gamma(\bar\mu_{t_0})
	\leq\liminf_{\varepsilon\to0}
	\Ent_\gamma(\bar\mu_{t_0}^\varepsilon),
	\qquad
	U_{t_0}\geq\limsup_{\varepsilon\to0}U_{t_0}^\varepsilon. 
\end{equation}
If equality holds at $t_0\in(0,1)$ for the limiting geodesic, then
$U_{t_0}=S(t_0)$, and consequently
\begin{equation}
	0\leq\limsup_{\varepsilon\to0}\Delta_\varepsilon(t_0)
	\leq U_{t_0}-S(t_0)=0. 
\end{equation}
Hence $\Delta_\varepsilon(t_0)\to0$. Let
$M_s^\varepsilon=\int|x|^2\,\d\bar\mu_s^\varepsilon$. Since
$M_i^\varepsilon\to \int|x|^2\,\d\bar\mu_i>0$ for $i=0,1$, every
sufficiently small $\varepsilon>0$ satisfies
\begin{equation}
	\begin{aligned}
		M_s^\varepsilon
		=
		(1-s)M_0^\varepsilon+sM_1^\varepsilon
		-s(1-s)W_\varepsilon^2
		\geq
		(1-s)^2M_0^\varepsilon+s^2M_1^\varepsilon>0, \qquad \frac{1}{3}\leq s\leq \frac{2}{3}. 
	\end{aligned}
\end{equation}
Combining \eqref{eq:entropy-gap-forcing} and
\eqref{eq:entropy-Green-representation}, we obtain
\begin{equation}
	\begin{aligned}
		\Delta_\varepsilon(t_0)&\geq \frac{1}{n^3}\int_0^1
		G_{\omega_\varepsilon}(t_0,s) e^{-\Ent_\gamma(\bar\mu_s^\varepsilon)/n}M_s^\varepsilon 
		\left|\frac{\d}{\d s} \Ent_\gamma(\bar\mu_s^\varepsilon)\right|^2\,\d s\\
		&\geq C \int_{1/3}^{2/3} \left|\frac{\d}{\d s} \Ent_\gamma(\bar\mu_s^\varepsilon)\right|^2\,\d s\\
		&\geq \frac{3C}{2}\int_{1/3}^{2/3} \int_{1/3}^{2/3} 
		\left|\frac{\d}{\d s} \Ent_\gamma(\bar\mu_s^\varepsilon)-\frac{\d}{\d r} \Ent_\gamma(\bar\mu_r^\varepsilon)\right|^2\,\d s\,\d r  
		\overset{\eqref{eq:centered-entropy-refined}}{\geq} \frac{C}{81} W_\varepsilon^4, 
	\end{aligned}
\end{equation}
where $C$ is independent of all sufficiently small $\varepsilon$. Since $\Delta_\varepsilon(t_0)\rightarrow0$ and $W_\varepsilon\to W_{\rm c}$, we obtain $W_{\rm c}=0$ and then $\mu_1$ is a translate of $\mu_0$. The converse proved above does not require smoothness. 

Finally, if $m_0=m_1$, then $m_t=m_0$,
$W_{\rm c}=W_2(\mu_0,\mu_1)$, and $\mathcal D_t
=
\Ent_\gamma(\mu_t)-\frac{1}{2}|m_0|^2$. 
Hence the same conclusions hold with
$\Ent_\gamma(\mu_t)$ in place of $\mathcal D_t$. In the equality case, the translation
vector must vanish, and hence $\mu_0=\mu_1$.
\end{proof}

We next turn to the functional application.  For a fixed $t\in(0,1)$,
$a,b\geq0$, and $p\in[0,\infty]$, define the weighted power mean
\[
 M_p^t(a,b)=
 \begin{cases}
 ((1-t)a^p+tb^p)^{1/p},&ab>0,\\
 0,&ab=0,
 \end{cases}
\]
with the usual limiting interpretations at $p=0$ and $p=\infty$.

\begin{theorem}[Gaussian BBL inequality with barycenter terms]
\label{thm:barycenter-BBL}
Fix $t\in(0,1)$ and $p\in[0,\infty]$. Let $f,g,h\geq0$ be
$\gamma$-integrable, with $f$ and $g$ log-concave, and set
$A=\int f\,\d\gamma>0$ and $B=\int g\,\d\gamma>0$.
Assume that
\[
 h((1-t)x+ty)\geq M_p^t(f(x),g(y))
 \qquad \text{for every }x,y\in\R^n.
\]
Define $\d\mu_f=\frac fA\,\d\gamma$ and 
$\d\mu_g=\frac gB\,\d\gamma$,
and let $b_f,b_g$ be the barycenters of $\mu_f,\mu_g$, respectively.
Denote $b_t=(1-t)b_f+tb_g$ and $r=p/(1+np)$, with $r=1/n$ when
$p=\infty$.  Then
\begin{equation}\label{eq:barycenter-BBL}
 \int h\,\d\gamma
 \geq e^{-|b_t|^2/2}
 M_r^t\!\left(Ae^{|b_f|^2/2},Be^{|b_g|^2/2}\right).
\end{equation}
In particular, if $b_f=b_g$, then
\begin{equation}\label{eq:same-barycenter-BBL}
 \int h\,\d\gamma\geq M_r^t(A,B).
\end{equation}
\end{theorem}

\begin{proof}
We modify the argument of \cite[Theorem~1.7]{AishwaryaLi} in two places
to account for nonzero barycenters.
Set $F=\log f$, $G=\log g$ and $H=\log h$ on their respective
positivity sets.  If $(X,Y)$ is an optimal coupling of $\mu_f$ and
$\mu_g$, and $\mu_t$ is the law of $(1-t)X+tY$, then Jensen's inequality
and Donsker--Varadhan duality yield 
\begin{equation}\label{eq:BBL-DV-reduction}
 \int h\,\d\gamma
 \geq
 M_p^t\!\left(e^{\int F\,\d\mu_f},e^{\int G\,\d\mu_g}\right)
 e^{-\Ent_\gamma(\mu_t)}.
\end{equation}
Truncating the logarithms in the Donsker--Varadhan variational formula
\cite[Section~2]{DonskerVaradhan} covers vanishing or unbounded functions, and the cases
$p=0,\infty$ follow by passage to the limit.

The measures $\mu_f$ and $\mu_g$ are $1$-strongly log-concave, and the
sine coefficients in \eqref{eq:intro-entropy-distortion} dominate their
affine counterparts. Therefore, Theorem~\ref{thm:centered-entropy}
gives
\begin{equation}
	e^{-\mathcal D(\mu_t)/n}
	\geq
	(1-t)e^{-\mathcal D(\mu_f)/n}
	+t e^{-\mathcal D(\mu_g)/n}, 
\end{equation}
which implies 
\begin{equation}\label{eq:entropy-power-mean}
	e^{-\mathcal D(\mu_t)}
	\geq M_{1/n}^t\!\left(
	e^{-\mathcal D(\mu_f)},e^{-\mathcal D(\mu_g)}
	\right).
\end{equation}
We use the following two-point form of the weighted H\"older
inequality:
\begin{equation}\label{eq:two-point-holder}
	M_p^t(a,b)M_{1/n}^t(c,d)
	\geq M_r^t(ac,bd),\qquad r=p/(1+np),
\end{equation}
which is the weighted H\"older inequality with
$1/r=1/p+n$; the cases $p=0$ and $p=\infty$ follow by
continuity.  Since the barycenter of $\mu_t$ is $b_t$, 
combining \eqref{eq:BBL-DV-reduction},
\eqref{eq:entropy-power-mean}, and \eqref{eq:two-point-holder} gives
\begin{equation}
	\int h\,\d\gamma
	\geq e^{-|b_t|^2/2}
	M_r^t\!\left(
	e^{\int F\,\d\mu_f-\mathcal D(\mu_f)},
	e^{\int G\,\d\mu_g-\mathcal D(\mu_g)}
	\right).
\end{equation}
Since $\d\mu_f=e^F A^{-1}\,\d\gamma$ and
$\d\mu_g=e^G B^{-1}\,\d\gamma$, we have
\begin{equation}
	e^{\int F\,\d\mu_f-\mathcal D(\mu_f)}
	=Ae^{|b_f|^2/2},
	\qquad
	e^{\int G\,\d\mu_g-\mathcal D(\mu_g)}
	=Be^{|b_g|^2/2}, 
\end{equation}
which proves \eqref{eq:barycenter-BBL}.
\end{proof}

\begin{corollary}[Gaussian BBL inequality with sine coefficients]
\label{cor:distortion-BBL}
Under the assumptions of Theorem~\ref{thm:barycenter-BBL}, denote
\[
 \omega=\sqrt{2/n}\left(W_2^2(\mu_f,\mu_g)-|b_f-b_g|^2\right)^{1/2}, 
\]
and
\[
 \alpha=\frac{\sin((1-t)\omega)}{\sin\omega},
 \qquad
 \beta=\frac{\sin(t\omega)}{\sin\omega},
\]
with $\alpha=1-t$ and $\beta=t$ when $\omega=0$. Then
\begin{equation}\label{eq:distortion-BBL}
 \int h\,\d\gamma
 \geq e^{-|b_t|^2/2}
 M_r^t\!\left(
 Ae^{|b_f|^2/2}\left(\frac{\alpha}{1-t}\right)^n,
 Be^{|b_g|^2/2}\left(\frac{\beta}{t}\right)^n
 \right).
\end{equation}
\end{corollary}

\begin{proof}
	The result follows by retaining the sine coefficients in
	\eqref{eq:intro-entropy-distortion} and applying \eqref{eq:two-point-holder} as in the
	proof of Theorem~\ref{thm:barycenter-BBL}.
\end{proof}

\begin{proof}[Proof of Corollary~\ref{cor:barycenter-BM}]
For $0<t<1$, apply Theorem~\ref{thm:barycenter-BBL} with
\begin{equation}
	 f=\mathbf1_{K_0},
	\qquad g=\mathbf1_{K_1},
	\qquad h=\mathbf1_{K_t},
	\qquad p=\infty. 
\end{equation}
The pointwise assumption follows from the convention
$M_p^t(a,b)=0$ when $ab=0$. Taking the $n$-th root of
\eqref{eq:barycenter-BBL} yields \eqref{eq:intro-barycenter-BM}.

For the equality case, apply Corollary~\ref{cor:distortion-BBL} to the
same indicator functions with $p=\infty$.  Taking the $n$-th root gives
the stronger estimate with sine coefficients
\begin{equation}
	e^{|b_t|^2/(2n)}\gamma(K_t)^{1/n}
	\geq \alpha A_0+\beta A_1,
	\qquad
	A_i=e^{|b_{K_i}|^2/(2n)}\gamma(K_i)^{1/n}>0. 
\end{equation}
If $\omega>0$, then $\alpha>1-t$ and $\beta>t$. Thus equality in
\eqref{eq:intro-barycenter-BM} forces $\omega=0$, and hence
$\gamma_{K_1}$ is a translate of
$\gamma_{K_0}$. Set $a=b_{K_1}-b_{K_0}$. 
Comparing their densities with respect to Lebesgue measure yields 
$a\cdot x=\mathrm{const}$ almost everywhere on the nonempty open set
$\operatorname{int}K_1$, which forces $a=0$. Thus $\gamma_{K_0}=\gamma_{K_1}$, and equality of their supports yields 
$\bar K_0=\bar K_1$.

Conversely, if $\bar K_0=\bar K_1$, then
$K_0,K_1$, and $K_t$ agree up to Gaussian-null boundaries, so equality
holds for every $t\in[0,1]$.
\end{proof}

\begin{corollary}[Gaussian Brunn--Minkowski inequality with sine coefficients]
	\label{cor:distortion-BM}
	Under the assumptions of
	Corollary~\ref{cor:barycenter-BM}, fix $t\in(0,1)$ and denote
	\[
	\omega=\sqrt{2/n}
	\left(
	W_2^2(\gamma_{K_0},\gamma_{K_1})
	-|b_{K_1}-b_{K_0}|^2
	\right)^{1/2},
	\]
	and 
	\[
	\alpha=\frac{\sin((1-t)\omega)}{\sin\omega},
	\qquad
	\beta=\frac{\sin(t\omega)}{\sin\omega},
	\]
	with $\alpha=1-t$ and $\beta=t$ when $\omega=0$. Then
	\begin{equation}\label{eq:distortion-BM}
		e^{|b_t|^2/(2n)}\gamma(K_t)^{1/n}
		\geq
		\alpha e^{|b_{K_0}|^2/(2n)}\gamma(K_0)^{1/n}
		+\beta e^{|b_{K_1}|^2/(2n)}\gamma(K_1)^{1/n}.
	\end{equation} 
	Moreover, equality in \eqref{eq:distortion-BM} at some
	$t\in(0,1)$ holds if and only if
			$\bar K_0=\bar K_1$. In this case, equality holds for every
	$t\in[0,1]$. 
\end{corollary}

\begin{proof}
	The result follows by applying Corollary~\ref{cor:distortion-BBL}
	with the indicator functions used in the proof of
	Corollary~\ref{cor:barycenter-BM}, setting $p=\infty$, and taking the
	$n$-th root of \eqref{eq:distortion-BBL}. 
	
	For the equality case, let $(\mu_s)_{s\in[0,1]}$ be the
	quadratic Wasserstein geodesic from $\gamma_{K_0}$ to
		$\gamma_{K_1}$. Since $\mu_t(K_t)=1$, nonnegativity of relative
		entropy yields $\Ent_\gamma(\mu_t)\geq-\log\gamma(K_t)$. Together
		with Theorem~\ref{thm:centered-entropy}, this gives
	\begin{equation}
		e^{|b_t|^2/(2n)}\gamma(K_t)^{1/n}
		\geq e^{-\mathcal D(\mu_t)/n}
		\geq \alpha A_0+\beta A_1,
		\qquad
		A_i=e^{|b_{K_i}|^2/(2n)}\gamma(K_i)^{1/n}>0. 
	\end{equation}
	Hence equality in \eqref{eq:distortion-BM} forces equality in
	\eqref{eq:intro-entropy-distortion}, which implies $\gamma_{K_1}$ is a
		translate of $\gamma_{K_0}$. The final density comparison and the
		converse are the same as in the proof of
		Corollary~\ref{cor:barycenter-BM}.
\end{proof}

We conclude with the functional inequalities obtained by joining a
measure to equilibrium.  This is the HWI principle of Otto and Villani
\cite{OttoVillani}, now applied with the centered entropy estimate.
The even versions of the resulting HWI, logarithmic Sobolev, and
Talagrand inequalities appear in \cite[Section~5]{AishwaryaRotem}; the
centering argument in the proof of Corollary~\ref{cor:centered-functionals}
removes the evenness assumption.
If $\d\mu=\rho\,\d\gamma$, denote
\[
 \FI_\gamma(\mu)=\int|\nabla\log\rho|^2\,\d\mu
\]
for relative Fisher information, with the usual value $+\infty$ outside
its form domain.

\begin{corollary}[Gaussian functional inequalities after centering]
\label{cor:centered-functionals}
Let $\mu$ be $1$-strongly log-concave, let $m=\int x\,\d\mu$, and denote
\[
 \mathcal D=\Ent_\gamma(\mu)-\frac12|m|^2,
 \quad
 \mathcal I=\FI_\gamma(\mu)-|m|^2,
 \quad
 \mathcal W^2=W_2^2(\mu,\gamma)-|m|^2.
\]
Then $\mathcal D\geq0$, $\mathcal I\geq0$, and
$0\leq\mathcal W^2\leq n$.  Moreover,
\begin{equation}\label{eq:centered-HWI}
 e^{\mathcal D/n}
 \leq
 \cos\!\left(\sqrt{\frac2n}\,\mathcal W\right)
 +\frac{\sin\!\left(\sqrt{2/n}\,\mathcal W\right)}{\sqrt{2n}}
  \sqrt{\mathcal I},
\end{equation}
\begin{equation}\label{eq:centered-LSI}
 4\mathcal D
 \leq2n\left(e^{2\mathcal D/n}-1\right)
 \leq\mathcal I,
\end{equation}
and
\begin{equation}\label{eq:centered-Talagrand}
 \mathcal W^2
 \leq-n\log\cos\!\left(\sqrt{\frac2n}\,\mathcal W\right)
 \leq\mathcal D.
\end{equation}
Moreover, for every $s\geq0$,
\begin{equation}\label{eq:two-rate-OU}
 \Ent_\gamma(P_s^*\mu)
 \leq e^{-4s}\left(\Ent_\gamma(\mu)-\frac12|m|^2\right)
      +\frac12e^{-2s}|m|^2.
\end{equation}
\end{corollary}

\begin{proof}
\medskip\noindent\textbf{Step 1: Centering.}
Let $\bar\mu=(x\mapsto x-m)_\#\mu$
be the centered translate of $\mu$. Set $\d\mu=\rho\,\d\gamma$ and
$\bar\rho=\d\bar\mu/\d\gamma$. Gaussian integration by parts and a
change of variables give
\begin{equation}
	\int \nabla\log\rho\,\d\mu=\int\nabla\rho\,\d\gamma=m,\qquad\log\bar\rho(x)
	=
	\log\rho(x+m)
	-m\cdot x
	-\frac{|m|^2}{2}, 
\end{equation}
where the Gaussian integration by parts in the first identity is
justified by standard radial cutoffs whenever
$\FI_\gamma(\mu)<\infty$. If $\FI_\gamma(\mu)=+\infty$, the identity for
the translated density implies
$\FI_\gamma(\bar\mu)=+\infty$, so the Fisher-information identity below
holds in the extended sense. 
Combining this with the orthogonal decomposition of the quadratic
transport cost, we obtain
\begin{equation}
	 \mathcal D=\Ent_\gamma(\bar\mu),
	\qquad
	\mathcal I=\FI_\gamma(\bar\mu),
	\qquad
	\mathcal W=W_2(\bar\mu,\gamma), 
\end{equation}
which yields the asserted nonnegativity. Hence it is sufficient
to assume $m=0$. 

Set $\omega=\sqrt{2/n}\,\mathcal W$. In the smooth case, let $T$ be
the Brenier map from $\gamma$ to $\bar\mu$. Since both measures are
centered, $\int(T(x)-x)\,\d\gamma(x)=0$, while Caffarelli's contraction theorem \cite{Caffarelli} yields 
$0\preceq DT\preceq I$. 
Applying the Gaussian Poincar\'e inequality
to each component of $T-I$ and summing over the components, we obtain 
\begin{equation}
	\mathcal W^2
	=\int|T(x)-x|^2\,\d\gamma(x)
	\leq\int\norm{DT-I}_{\HS}^2\,\d\gamma
	\leq n.
\end{equation}
The general case follows by Lemma~\ref{lem:regularization}. In particular,
$0\leq\omega\leq\sqrt2<\pi/2$. 

\medskip\noindent\textbf{Step 2: HWI, logarithmic Sobolev, and Talagrand estimates.}
First assume that $\bar\mu$ has a smooth potential with two-sided
Hessian bounds. Apply Theorem~\ref{thm:centered-entropy} to the
geodesic from $\bar\mu$ to $\gamma$, and write
$E_t=\Ent_\gamma(\mu_t)$. 
Differentiating the finite-time inequality at the first endpoint and using
$-E_0'\leq\sqrt{\FI_\gamma(\bar\mu)}\,W_2(\bar\mu,\gamma)$ yields 
\eqref{eq:centered-HWI}, as in
\cite[Theorem~5.5]{AishwaryaRotem}. 
 Squaring this inequality and
arguing as in \cite[Corollary~5.6]{AishwaryaRotem} yields
\eqref{eq:centered-LSI}. Reversing the geodesic and differentiating at
the Gaussian endpoint, where the Fisher information vanishes, yields $e^{-\mathcal D/n}\leq\cos\omega$. 
Together with $-\log\cos\omega\geq\omega^2/2$, this yields
\eqref{eq:centered-Talagrand}, as in
\cite[Corollary~5.11]{AishwaryaRotem}. Once
Theorem~\ref{thm:centered-entropy} has replaced the even entropy
estimate used there, none of these arguments requires symmetry. 

Since $P_s^*\bar\mu$ remains centered and $1$-strongly log-concave,
the de Bruijn argument as in \cite[Corollary~5.9]{AishwaryaRotem} applied with
\eqref{eq:centered-LSI} yields 
\begin{equation}
	\Ent_\gamma(P_s^*\bar\mu)
	\leq e^{-4s}\Ent_\gamma(\bar\mu).
\end{equation}
The mean of $P_s^*\mu$ is $e^{-s}m$, and its centered translate is
$P_s^*\bar\mu$. The Gaussian entropy translation formula therefore
yields \eqref{eq:two-rate-OU}. 

For the general case, set
$\mu_\varepsilon=P_\varepsilon^*\bar\mu$.
By Lemma~\ref{lem:regularization}, applied with parameter
$e^{2\varepsilon}-1$, these measures satisfy the required
smooth two-sided Hessian bounds and converge to $\bar\mu$ in $W_2$.
Moreover, the data processing inequality, lower semicontinuity, and the
Ornstein--Uhlenbeck gradient commutation give
\begin{equation}
	\Ent_\gamma(\mu_\varepsilon)
	\longrightarrow\Ent_\gamma(\bar\mu),
	\qquad
	\FI_\gamma(\mu_\varepsilon)
	\leq e^{-2\varepsilon}\FI_\gamma(\bar\mu). 
\end{equation}
These facts allow passage to the limit in \eqref{eq:centered-HWI}, \eqref{eq:centered-LSI} and \eqref{eq:centered-Talagrand}.
For \eqref{eq:two-rate-OU}, use
$P_s^*\mu_\varepsilon=P_{s+\varepsilon}^*\bar\mu$ and the lower
semicontinuity of entropy before letting $\varepsilon\to0$.
\end{proof}

\begin{remark}
	The two rates in \eqref{eq:two-rate-OU} are optimal.  Translates of
	$\gamma$ attain the $e^{-2s}$ term, while centered Gaussian contractions
	with covariance approaching $I$ show that $e^{-4s}$ cannot be improved
		for centered measures.
	For a centered $1$-strongly log-concave measure, the result of Bobkov,
	Gozlan, Roberto, and Samson \cite[Corollary 2.2]{BobkovEtAl} gives
	\[
	n\left(e^{2\Ent_\gamma(\mu)/n}-1\right)
	\leq \FI_\gamma(\mu).
	\]
	Under the same centering and curvature assumptions,
	\eqref{eq:centered-LSI} doubles the left-hand side.  This improvement
	uses the full geodesic Poincar\'e estimate.
\end{remark}

\appendix

\section{Failure of interpolation by the endpoint Poincar\'e constants}
\label{app:actual-CP-counterexample}

\begin{proposition}
	\label{prop:actual-CP-counterexample}
	There exist $\kappa>0$ and two $\kappa$-strongly log-concave
	probability measures $\mu_0,\mu_1$ on $\R^2$, such that their Wasserstein
	midpoint satisfies 
	\begin{equation}
			\sqrt{C_P(\mu_{1/2})}
		>\frac{1}{2}\sqrt{C_P(\mu_0)}+\frac{1}{2}\sqrt{C_P(\mu_1)}. 
	\end{equation}
\end{proposition}

\begin{proof}
	\medskip\noindent\textbf{Step 1: The square-shear counterexample.}
	Let
	\[
	\begin{gathered}
	\Omega=\left(-\frac{\pi}{2},\frac{\pi}{2}\right)^2\subseteq\R^2,
	\qquad S=\begin{pmatrix}0&1\\1&0\end{pmatrix},\\
	B_s=\begin{pmatrix}1&s\\s&1\end{pmatrix},
	\qquad \nu_s=(B_s)_\#m_\Omega.
	\end{gathered}
	\]
	where $m_\Omega$ is normalized Lebesgue measure on $\Omega$. Since the restrictions of $C_c^\infty(\R^2)$ to $\Omega$ are dense in
	$H^1(\Omega)$, pulling the Rayleigh quotient back by $B_s$ yields 
	\begin{equation}\label{eq:counterexample-rayleigh}
		C_P(\nu_s)^{-1}
		=
		\lambda(s)
		:=
		\inf_{\substack{
				u\in H^1(\Omega)\setminus\{0\}\\
				\int_\Omega u\,\d m_\Omega=0}}
		\frac{\int_\Omega
			\ip{B_s^{-2}\nabla u}{\nabla u}\,\d m_\Omega}
		{\int_\Omega u^2\,\d m_\Omega}. 
	\end{equation}
	
	Note that $S^2=I$ and 
	\begin{equation}
			B_s^{-2}=I-2sS+3s^2I+O(s^3), \qquad s\rightarrow0^+. 
	\end{equation}
	Consequently, for $u,v\in H^1(\Omega)$,
	\begin{equation}\label{eq:square-form-expansion}
	\begin{aligned}
	\int_\Omega\ip{B_s^{-2}\nabla u}{\nabla v}\,\d m_\Omega
	=\int_\Omega\ip{\nabla u}{\nabla v}\,\d m_\Omega
	&-2s\int_\Omega
	\bigl(\partial_1u\,\partial_2v+
	      \partial_2u\,\partial_1v\bigr)\,\d m_\Omega\\
	&+3s^2\int_\Omega
	\ip{\nabla u}{\nabla v}\,\d m_\Omega+O(s^3).
	\end{aligned}
	\end{equation}
	
	The numerator in \eqref{eq:counterexample-rayleigh} defines a
	real-analytic family of closed symmetric forms, with common form domain
	$H^1(\Omega)$ and compact resolvent. We next determine the expansion of
	$\lambda(s)$ as $s\rightarrow0$.
	At $s=0$, separation of variables shows that the first positive
	Neumann eigenvalue on $\Omega$ is $1$, with eigenspace
	\begin{equation}
		E=\operatorname{span}\{e_1,e_2\},
		\qquad
		e_1(x,y)=\sqrt2\sin x,
		\qquad
		e_2(x,y)=\sqrt2\sin y, 
	\end{equation}
	and the next eigenvalue is $2$. In the orthonormal basis $(e_1,e_2)$,
	the restriction to $E$ of the first-order form in
	\eqref{eq:square-form-expansion} has matrix $-16S/\pi^2$. Its two
	eigenvalues are distinct. The lower one is $-16/\pi^2$, with normalized
	eigenfunction $u_0=(e_1+e_2)/\sqrt2=\sin x+\sin y$.
	
	Analytic perturbation theory for closed symmetric forms
	\cite[Chapter~VII, Section~4.6]{KatoPerturbation} therefore gives two
	analytic eigenvalue branches issuing from the double eigenvalue $1$.
	Because their first derivatives at $0$ are distinct, the lower branch
	is simple for all sufficiently small $s>0$ and equals $\lambda(s)$.
	Thus
	\begin{equation}\label{eq:counterexample-eigenvalue-expansion}
			\lambda(s)
		=
		1-\frac{16}{\pi^2}s+\beta s^2+O(s^3),
		\qquad s\rightarrow0^+. 
	\end{equation}
	
	We claim that $\beta >\frac{192}{\pi^4}$. Indeed, let $I=(-\pi/2,\pi/2)$ and $\d m_I=\d x/\pi$, and consider the
	normalized Neumann basis 
	\begin{equation}
			\phi_0=1,
		\qquad
		\phi_k(x)=
		\begin{cases}
			\sqrt2\sin(kx),&k\geq1\text{ odd},\\
			\sqrt2\cos(kx),&k\geq2\text{ even}. 
		\end{cases} 
	\end{equation}
	Then $\phi_m\otimes\phi_n$ is an orthonormal Neumann eigenbasis on
	$\Omega$, with eigenvalue $m^2+n^2$. Set 
	\[
	A_m=\int_I\cos x\,\phi_m(x)\,\d m_I(x),
	\qquad
	D_m=\int_I\phi_m'(x)\,\d m_I(x). 
	\]
	Then 
	\begin{equation}
		A_m^2=
		\begin{cases}
			4/\pi^2,&m=0,\\
			8/\bigl(\pi^2(m^2-1)^2\bigr),&m\geq2\text{ even},\\
			0,&m\text{ odd},
		\end{cases}
		\qquad
		D_m^2=
		\begin{cases}
			8/\pi^2,&m\text{ odd},\\
			0,&m\text{ even}.
		\end{cases} 
	\end{equation}
	
	Let $q_1$ and $q_2$ denote the coefficients of $s$ and $s^2$ in
	\eqref{eq:square-form-expansion}. Then $q_2(u_0,u_0)=3$, while
	\[
	 q_1(u_0,\phi_m\otimes\phi_n)
	 =-2(A_mD_n+D_mA_n).
	\]
	Because the lower eigenvalue of $q_1|_E$ is simple, the second-order
	formula for a split eigenvalue applies. Indeed, if
	$\alpha=-16/\pi^2$, then
	$q_1(u_0,v)=\alpha\ip{u_0}{v}_{L^2(m_\Omega)}$ for every $v\in E$;
	hence the forcing term in the first-order eigenvector equation is
	orthogonal to $E$. The formula is therefore
	\begin{equation}\label{eq:counterexample-second-order-formula}
	 \beta
	 =3-
	 \sum_{\substack{m,n\geq0\\m^2+n^2>1}}
	 \frac{4(A_mD_n+D_mA_n)^2}{m^2+n^2-1}.
	\end{equation}
	Using the parity relations above in
	\eqref{eq:counterexample-second-order-formula}, we obtain
	\begin{equation}
	\begin{aligned}
		\beta
		&=3-8A_0^2
		\sum_{\substack{n\geq3\\n\ {\rm odd}}}
		\frac{D_n^2}{n^2-1}
		-8 \sum_{\substack{m\geq2\ {\rm even}\\n\geq1\ {\rm odd}}}
		\frac{A_m^2D_n^2}{m^2+n^2-1}\\
		&=3-\frac{64}{\pi^4}-\frac{512}{\pi^4}D,
	\end{aligned}
	\end{equation}
	where 
	\begin{equation*}
	\begin{aligned}
			D:=
		\sum_{\substack{m\geq2\ {\rm even}\\ n\geq1\ {\rm odd}}}
		\frac1{(m^2-1)^2(m^2+n^2-1)}&=
		\frac{\pi}{4}\sum_{k=1}^\infty
		\frac{\tanh\!\left(\frac{\pi}{2}\sqrt{4k^2-1}\right)}
		{(4k^2-1)^{5/2}}\\
		&<\frac{\pi}{4}\sum_{k=1}^{\infty}\frac{1}{3^{5/2}k^5}
		<\frac{1}{16}.
	\end{aligned}
	\end{equation*}
	Consequently,
	\begin{equation}
			\beta>3-\frac{96}{\pi^4}>\frac{192}{\pi^4}, 
	\end{equation}
	which proves the claim. 
	
	Therefore, by \eqref{eq:counterexample-rayleigh} and
	\eqref{eq:counterexample-eigenvalue-expansion}, for all sufficiently
	small $r>0$,
	\begin{equation}
			\sqrt{C_P(\nu_{r/2})}-\left(\frac{1}{2}\sqrt{C_P(\nu_0)}+\frac{1}{2}\sqrt{C_P(\nu_r)}\right)
		=
		\left(\frac{\beta}{8}-\frac{24}{\pi^4}\right)r^2+O(r^3)>0. 
	\end{equation}
	
\medskip\noindent\textbf{Step 2: Strongly log-concave approximation.}
For $\delta>0$, set
\[
\d\eta_\delta(x)
=
Z_\delta^{-1}e^{-\delta|x|^2/2}
\mathbf1_\Omega(x)\,\d x,
\qquad
\eta_{\delta,s}=(B_s)_\#\eta_\delta.
\]
Here $Z_\delta$ is the normalizing constant. Up to an additive
constant, the potential of $\eta_{\delta,s}$ is
$V_{\delta,s}(y)=\frac{\delta}{2}|B_s^{-1}y|^2+
\iota_{B_s\bar\Omega}(y)$, where $\iota_K$ is zero on $K$ and
$+\infty$ on $K^c$. Since $\iota_{B_s\bar\Omega}$ is convex and
\begin{equation}
	B_s^{-2}\succeq\frac1{(1+r)^2}I,
	\qquad 0\leq s\leq r, 
\end{equation}
the measures $\eta_{\delta,s}$ are
$\delta(1+r)^{-2}$-strongly log-concave. Moreover,
$\frac{\d\eta_\delta}{\d m_\Omega}\in
[e^{-\delta\pi^2/4},e^{\delta\pi^2/4}]$. Hence, by
\eqref{eq:counterexample-rayleigh},
\begin{equation}
	e^{-\delta\pi^2/2}C_P(\nu_s)
	\leq C_P(\eta_{\delta,s})
	\leq e^{\delta\pi^2/2}C_P(\nu_s),
	\qquad 0\leq s\leq r, 
\end{equation}
which implies $C_P(\eta_{\delta,s})\to C_P(\nu_s)$ uniformly on $[0,r]$. Thus, for $\delta>0$ sufficiently small, 
\begin{equation}\label{eq:strongly-log-concave-counterexample}
	\sqrt{C_P(\eta_{\delta,r/2})}
	>\frac{1}{2}\sqrt{C_P(\eta_{\delta,0})}+\frac{1}{2}\sqrt{C_P(\eta_{\delta,r})}. 
\end{equation}

Fix such a $\delta>0$. Since $B_r$ is symmetric positive definite, it is
the Brenier map from $\eta_{\delta,0}$ to $\eta_{\delta,r}$. Moreover, $(1-t)I+tB_r=B_{tr}$, 
so their Wasserstein geodesic is
$(\eta_{\delta,tr})_{0\leq t\leq1}$. Taking
$\mu_0=\eta_{\delta,0}$ and $\mu_1=\eta_{\delta,r}$ gives
\begin{equation}
	\sqrt{C_P(\mu_{1/2})}>\frac{1}{2}\sqrt{C_P(\mu_0)}+\frac{1}{2}\sqrt{C_P(\mu_1)}, 
\end{equation}
which proves the proposition.
\end{proof}

\bigskip

\noindent\textbf{Declaration.}
The authors declare no competing interests. No datasets were generated or analyzed in this study. During the preparation of this manuscript, the authors used OpenAI Codex to assist with the organization and editing of the manuscript. The authors take full responsibility for the final content. 

\medskip

\noindent\textbf{Funding.}
This work was supported in part by the National Key R\&D Program of
China (2021YFA1000900, 2021YFA1002200), the National Natural Science
Foundation of China (12201596), the Shandong Provincial Natural Science
Foundation (ZR2025QB05), and the Taishan Scholars Program of Shandong
Province (tsqn202408059).
\end{document}